\documentclass[12pt]{article}
\title{$C$-minimal fields have the exchange property}
\author{Will Johnson}

\usepackage{amsmath, amssymb, amsthm}    	
\usepackage{fullpage} 	
\usepackage{amscd}
\usepackage{hyperref}
\usepackage[all]{xy}
\usepackage[T1]{fontenc}
\usepackage{lmodern}
\usepackage{centernot}
\usepackage[shortlabels]{enumitem}

\DeclareMathOperator*{\ind}{\raise0.2ex\hbox{\ooalign{\hidewidth$\vert$\hidewidth\cr\raise-0.9ex\hbox{$\smile$}}}}

\newcommand{\Shape}{\operatorname{Shape}}

\newcommand{\RV}{\mathrm{RV}}
\newcommand{\rv}{\operatorname{rv}}

\newcommand{\End}{\operatorname{End}}

\newcommand{\characteristic}{\operatorname{char}}

\newcommand{\res}{\operatorname{res}}

\newcommand{\id}{\operatorname{id}}

\newcommand{\acl}{\operatorname{acl}}

\newcommand{\dom}{\operatorname{dom}}
\newcommand{\im}{\operatorname{im}}

\newcommand{\Stab}{\operatorname{Stab}}

\newcommand{\Mov}{\operatorname{Mov}}

\newcommand{\dpr}{\operatorname{dp-rk}}

\newtheorem{theorem}{Theorem}[section] 
\newtheorem{slogan}{Slogan}

\newtheorem{lemma}[theorem]{Lemma}

\newtheorem{corollary}[theorem]{Corollary}
\newtheorem{fact}[theorem]{Fact}

\newtheorem{question}[theorem]{Question}
\newtheorem{proposition}[theorem]{Proposition}
\newtheorem{proposition-eh}[theorem]{Proposition(?)}
\newtheorem*{theorem-star}{Theorem}
\newtheorem*{conjecture-star}{Conjecture}
\newtheorem*{lemma-star}{Lemma}
\newtheorem{claim}[theorem]{Claim}
\newtheorem{observation}[theorem]{Observation}

\theoremstyle{definition}
\newtheorem{definition}[theorem]{Definition}
\newtheorem{example}[theorem]{Example}

\newtheorem{remark}[theorem]{Remark}

\theoremstyle{remark}

\newtheorem*{acknowledgment}{Acknowledgments}

\newcommand{\Qq}{\mathbb{Q}}

\newcommand{\err}{\mathrm{err}}

\newcommand{\Zz}{\mathbb{Z}}

\newcommand{\Mm}{\mathbb{M}}

\newcommand{\Ll}{\mathcal{L}}
\newcommand{\Oo}{\mathcal{O}}
\newcommand{\mm}{\mathfrak{m}}

\newcommand{\Lr}{\Ll_{rings}}
\newcommand{\kx}{{k^\times}}

\newcommand{\bc}{{\bar{c}}}

\newenvironment{claimproof}[1][\proofname]
               {
                 \proof[#1]
                 
               }
               {
                 \endproof
               }

\begin{document}

\maketitle

\begin{abstract}
  We show that $C$-minimal fields (i.e., $C$-minimal expansions of ACVF)
  have the exchange property, answering a question of Haskell and
  Macpherson \cite{cminfields}.  Additionally, we strengthen some
  theorems of Cubides Kovacsics and Delon \cite{CK-D} on $C$-minimal
  fields.  First, we show that definably complete $C$-minimal fields of
  characteristic 0 have generic differentiability.  Second, we show
  that if the induced structure on the residue field is a pure ACF,
  then polynomial boundedness holds.  In fact, polynomial boundedness
  can only fail if there are unexpected definable automorphisms of the
  multiplicative group of the residue field.
\end{abstract}

\section{Introduction}
Recall from \cite{c-source,cminfields} that a \emph{$C$-minimal field}
is a structure $(K,\Oo,+,\cdot,\ldots)$ expanding a non-trivial valued
field $(K,\Oo)$, such that
\begin{itemize}
\item Every unary definable set $D \subseteq K^1$ is a finite boolean
  combination of balls and points.
\item This property continues to hold in elementary extensions $K'
  \succeq K$.
\end{itemize}
For example, the theory ACVF of algebraically closed valued fields is
$C$-minimal \cite[Theorem~4.11]{c-source}.  Conversely, Haskell and
Macpherson showed that every $C$-minimal field is an expansion of ACVF
\cite[Theorem~C]{cminfields}.  Macpherson and Steinhorn consider a more general notion of
\emph{$C$-structures} and \emph{$C$-minimal theories}, which we will not
precisely define here.

$C$-minimality is meant to be an analog of o-minimality.  One of the key properties of (dense) o-minimal theories is the \emph{exchange property}
\begin{equation*}
  c \in \acl(Ab) \setminus \acl(A) \implies b \in \acl(Ac) \text{ for } A \subseteq M, ~ b,c \in M.
\end{equation*}
The exchange property ensures that o-minimal theories are ``geometric
theories'' in the sense of Hrushovski and Pillay
\cite[Section~2]{udi-anand-group-field}, and are therefore endowed
with an excellent dimension theory.

In contrast to the o-minimal situation, the exchange property can fail
in a general $C$-minimal theory \cite[Example~3.2]{c-source}.  Haskell
and Macpherson asked whether the exchange property holds in $C$-minimal
fields, i.e., $C$-minimal expansions of ACVF
\cite[Problem~6.4]{cminfields}.  We give a positive answer to this
question:
\begin{theorem-star}[= Theorem~\ref{big-thm}]
  Let $T$ be a $C$-minimal expansion of ACVF.  Then $T$ has the exchange
  property.
\end{theorem-star}
Consequently, $C$-minimal fields are geometric and have an excellent
dimension theory.

The technique used to prove Theorem~\ref{big-thm} has some other
applications to $C$-minimal fields, allowing us to strengthen some of
the main results of Cubides Kovacsics and Delon \cite{CK-D}.

The first of these applications concerns polynomial boundedness.
Recall that $K$ is \emph{polynomially bounded} if for any definable
function $f : K \to K$, there is an integer $n$ such that for all $x$
in a neighborhood of $\infty$,
\begin{equation*}
  v(f(x)) \ge n \cdot v(x) \qquad \left(\text{or in multiplicative notation,
  } |f(x)| \le |x|^n\right).
\end{equation*}
It is an open question whether all $C$-minimal fields are polynomially
bounded \cite[Question~1]{CK-D}.

Let $\Gamma$ and $k$ be the value group and residue field of $K$.  Say
that $\Gamma$ has an \emph{exotic automorphism} if there is a
definable automorphism of $(\Gamma,+)$ other than the automorphisms
$f(x) = q \cdot x$ for $q \in \Qq^\times$.  If the induced structure on $(\Gamma,+)$ is a pure divisible ordered abelian group (DOAG), then $\Gamma$ has no exotic automorphisms.

Similarly, say that the multiplicative group $\kx$ has an \emph{exotic
automorphism} if there is a definable automorphism of $\kx$ other than
the automorphisms $f(x) = x^{\pm p^n}$, where $p$ is the
characteristic exponent of $k$.  If the induced structure on $k$ is a
pure ACF, then $\kx$ has no exotic automorphisms.
\begin{theorem-star}[= Theorem~\ref{pbound}]
  Let $K$ be a $C$-minimal field.
  \begin{enumerate}
  \item If $\Gamma$ has no exotic automorphisms, then $K$ is
    polynomially bounded.
  \item If $\kx$ has no exotic automorphisms, then $K$ is polynomially
    bounded.
  \end{enumerate}
  In particular, if the induced structure on $\Gamma$ is a pure DOAG,
  or the induced structure on $k$ is a pure ACF, then $K$ is
  polynomially bounded.
\end{theorem-star}
Part (1) was previously proved by Cubides Kovacsics and Delon
\cite[Theorems~4.3+4.7]{CK-D}, but (2) is new.  It is unknown whether
exotic automorphisms of $\kx$ can exist.

There are also some consequences for definable completeness.  Recall
that $K$ is \emph{definably complete} if whenever $\mathcal{C}$ is a
definable chain of balls, and the radii of the balls in $\mathcal{C}$
approach $\infty$ (in multiplicative notation: approach $0$), then
$\bigcap \mathcal{C}$ is non-empty.  Most $C$-minimal theories arising
in practice are definably complete, though Delon
\cite[Th\'eor\`eme~5.4]{honneur} has constructed examples of $C$-minimal
fields which are not.
\begin{theorem-star}[= Theorem~\ref{limits-exist}]
  Let $K$ be a definably complete $C$-minimal field.  If $f : X \to K$
  is a definable function and $X$ is a punctured neighborhood of 0,
  then $\lim_{x \to 0} f(x)$ exists in $K \cup \{\infty\}$.
\end{theorem-star}
This strengthens \cite[Lemma~5.3]{CK-D}, removing their assumption
that definable functions $f : \Gamma \to \Gamma$ are eventually
linear.
\begin{theorem-star}[= Corollary~\ref{asymptotics-ckd}]
  Let $K$ be a definably complete $C$-minimal field of characteristic 0.
  Suppose that $\Gamma$ has no exotic automorphisms, or $\kx$ has no
  exotic automorphisms.  If $f : X \to K$ is a definable function on a
  punctured neighborhood of 0, then $f(x)$ has an asymptotic expansion
  \begin{equation*}
    f(x) \sim \sum_{n = 1}^\infty c_n x^{m_n} \text{ as } x \to 0,
  \end{equation*}
  for some $c_1, c_2, \ldots \in K$ and integers $m_1 < m_2 <
  \cdots$.
\end{theorem-star}
This strengthens \cite[Theorem~6.1]{CK-D}, weakening their assumption
that definable functions $f : \Gamma \to \Gamma$ are eventually
$\Qq$-linear.
\begin{theorem-star}[= Theorem~\ref{diff-thm}]
  Let $K$ be a definably complete $C$-minimal field of characteristic 0.
  If $f : K \to K$ is definable, then $f$ is differentiable at all but
  finitely many points.
\end{theorem-star}
This strengthens (part of) \cite[Theorem~6.3]{CK-D}, removing their
assumption that definable functions $f : \Gamma \to \Gamma$ are
eventually $\Qq$-linear.
With more work, one can also get generic differentiability in several
variables (Theorem~\ref{diff-thm2}) and an implicit/inverse function
theorem (Theorem~\ref{inverse}).

\subsection{Proof strategy}
The core technique is an analysis of definable subsets of $\RV \times
\RV$ in Section~\ref{sec-heart}.  In
Theorems~\ref{key-thm}--\ref{dpr1}, we show that if $D \subseteq \RV
\times \RV$ is definable and has dp-rank 1, then the image of $D$ in
$\Gamma \times \Gamma$ is a finite union of line segments (graphs of
linear functions).

In the applications, we usually apply this to the following setting.
Let $f : K \to \Gamma \times \Gamma$ be a definable function.
Suppose that $f$ lifts to a map $\tilde{f} : K \to \RV \times \RV$:
\begin{equation*}
  \xymatrix{K \ar@{-->}[r]^-{\tilde{f}} \ar[dr]_-f & \RV \times \RV \ar[d] \\ & \Gamma \times \Gamma}
\end{equation*}
Then $\im(\tilde{f})$ has dp-rank 1, so $\im(f)$ is a finite union of
line segments.  Many natural functions from $K$ to $\Gamma$ lift to
$\RV$ (for example, see \cite[Section 5]{CK-D}).  Consequently, many
of the definable functions $\Gamma \to \Gamma$ appearing in
\cite{CK-D} are automatically linear or eventually linear.  This
allows us to drop ``eventual linearity'' assumptions from many of the
results in \cite{CK-D}.

The connection to the exchange property is as follows.  Suppose the
exchange property fails.  By a theorem of Haskell and Macpherson
\cite[Proposition~6.1]{cminfields}, there must be a definable
bijection $f$ from an open set in $K$ to an antichain in the tree of
closed balls.  One can arrange for $f$ to land in the antichain
$K/\Oo$ of closed balls of valuative radius 0.  The set $K/\Oo$ is
itself a dense ultrametric space with ``value set'' $\Gamma_{< 0}$.
Restricting $f$, one can arrange for $f$ to be a homeomorphism between
an open set in $K$ and an open set in $K/\Oo$.  By something that is
morally the ``$\Gamma$-factorization II'' of \cite[Theorem~3.3]{CK-D}, one
can shrink $\dom(f)$ and arrange that
\begin{equation*}
  v(f(x) - f(y)) = g(v(x-y))
\end{equation*}
for some definable function $g : \Gamma \to \Gamma_{<0}$.  Since $f$ is a
homeomorphism, we must have $x \approx y \implies f(x) \approx f(y)$,
which means that
\begin{equation*}
  \lim_{\gamma \to +\infty} g(\gamma) = \sup \Gamma_{< 0} = 0. \tag{$\ast$}
\end{equation*}
But $g$ can be lifted to the RV-sorts (compare with
\cite[Theorem~5.8]{CK-D}), so $g$ is eventually linear by the results of
Section~\ref{sec-heart}.  No linear function $g$ can satisfy ($\ast$), unless $g$ is the constant 0, contradicting the fact that $g$'s codomain is $\Gamma_{<0}$.

The technical analysis of $\RV \times \RV$ in Section~\ref{sec-heart}
is based on the tension between the following two properties of the
short exact sequence $1 \to \kx \to \RV \to \Gamma \to 1$:
\begin{itemize}
\item The two groups $\kx$ and $\Gamma$ are orthogonal to each other.
\item All three groups have dp-rank 1, so the sequence does not
  definably split.
\end{itemize}
The other key ingredients are the fact that $\kx$ and $\Gamma$ are
stably embedded and eliminate imaginaries, $\kx$ is strongly minimal,
and $\Gamma$ is o-minimal.  Using all these ingredients, we show that
the induced structure on $\RV$ is highly constrained.

\subsection{Hensel minimality and its relatives}
The \emph{Hensel minimal} fields of Cluckers, Halupczok, and
Rideau-Kikuchi \cite{hens-min1,hens-min2} satisfy generic
differentiability and the exchange property \cite[Lemma~5.3.5,
  Corollary~3.2.7]{hens-min1}.  There is some overlap between hensel
minimality and $C$-minimality.  For example, ACVF$_0$ is hensel
minimal.  It is natural to ask whether the exchange property and
generic differentiability results (Theorems~\ref{big-thm},
\ref{diff-thm}) follow from the results on hensel minimality.  For
example,
\begin{quote}
  If $K$ is a definably complete $C$-minimal field, is $K$ hensel
  minimal?
\end{quote}
Hensel minimality is only defined on fields of characteristic zero, so
we should restrict to the case where $K$ has characteristic 0.  Even
then, the answer to this question is \textsc{no}.
\begin{example}
  Suppose $K$ has residue characteristic 0, so we are in the setting
  of \cite{hens-min1} rather than \cite{hens-min2}.  There are several
  forms of hensel minimality.  The weakest form, 0-h-minimality,
  implies for expansions of algebraically closed valued fields that
  every 0-definable closed ball $B \subseteq K$ intersects
  $\acl(\varnothing)$.\footnote{Apply the definition of 0-h-minimality
  \cite[Definition~1.2.3]{hens-min1} with $\lambda = 1$, $A =
  \varnothing$, and $X = B$.  There should be a finite
  $\varnothing$-definable set $C \subseteq K$ which ``1-prepares'' $B$
  in the sense of \cite[Definition~1.2.2]{hens-min1}.  Under the
  assumption that the value group is densely ordered and the residue
  field is infinite, one can check that a closed ball $B$ is
  1-prepared by a finite set $C$ if and only if $C$ intersects $B$.}
  We can produce a $C$-minimal expansion of ACVF$_{0,0}$ in which this
  property fails by adding a unary predicate symbol to name a random
  ball $B$.  If $B$ is chosen correctly, $B$ will not intersect
  $\acl(\ulcorner B \urcorner)$.  The resulting expansion of ACVF will
  be $C$-minimal and definably complete, but not 0-h-minimal.
\end{example}
Another related concept is the \emph{V-minimality} of Hrushovski and
Kazhdan \cite[Section~3.4]{hk}.  In the V-minimal setting, Hrushovski
and Kazhdan prove generic differentiability and the exchange property
\cite[Corollary~5.17, Lemma~3.57]{hk}.  A V-minimal theory is a
definably complete $C$-minimal expansion of ACVF$_{0,0}$ satisfying some
additional assumptions, one of which is that $\RV$ has no induced
structure beyond what is present in ACVF$_{0,0}$.  Consequently,
V-minimal theories have no exotic automorphisms on $\Gamma$ or $\kx$,
and all the results above apply to them.  Our results generalize the
generic differentiability and exchange property for V-minimal
theories.

In some sense, our setting is orthogonal to V-minimality.  Most of our
work goes into analyzing the induced structure on the $\RV$ sort
(Section~\ref{sec-heart}).  But V-minimality assumes that the induced
structure on $\RV$ is pure, sidestepping this issue.

\section{Basic facts}
We review some fundamental facts which will be used in the more
technical Section~\ref{sec-heart}.  \emph{All} of these facts are
surely known to experts.
\subsection{Consequences of $C$-minimality}
Let $T$ be a complete $C$-minimal expansion of ACVF.  Assume that $T$
includes a constant symbol $t$ in the home sort such that $v(t) > 0$.
Work in a monster model $\Mm$ of $T$.  Let $k$ and $\Gamma$ be the
residue field and value group.
\begin{fact}[{= \cite[Proposition 4.4, Principe 1]{honneur}}] \label{inf-fact}
  Let $D \subseteq \Mm$ be definable.  Then there is a ball $B \ni 0$
  such that $D$ contains or is disjoint from $B \setminus \{0\}$.
\end{fact}
This is well-known---it is essentially the statement that there is a
unique non-zero infinitesimal type around 0.  Nevertheless, we sketch
the proof for completeness.
\begin{proof}[Proof sketch]
  Say that $D$ is ``good'' if $D$ satisfies the conclusion of
  Fact~\ref{inf-fact}.  It is straightforward to check that every ball
  or singleton is good, and that boolean combinations of good sets are
  good.  By $C$-minimality, every definable set is good.
\end{proof}
\begin{corollary} \label{inf-cor}
  Let $\{D_1,\ldots,D_n\}$ be a partition of $\Mm^\times$ into
  finitely many definable sets.  Then there is an $i \le n$ and a ball
  $B \ni 0$ such that $D_i \supseteq B \setminus \{0\}$.
\end{corollary}
\begin{proof}
  Otherwise, for each $i$ there is a ball $B_i \ni 0$ such that $D_i
  \cap B_i = \varnothing$.  The intersection $\bigcap_{i = 1}^n B_i$
  is then a ball around 0 disjoint from $\bigcup_{i = 1}^n D_i =
  \Mm^\times$.  No such ball exists.
\end{proof}
We next review some facts about the
induced structure on $k$ and $\Gamma$.  Many of these can be found in
\cite{cminfields,CK-D}.
\begin{fact} \phantomsection \label{gamma-induce}
  \begin{enumerate}
  \item \label{gi1} The value group $\Gamma$ is a divisible ordered
    abelian group.
  \item \label{gi2} The induced structure on $\Gamma$ is o-minimal.
  \item \label{gi3} $\Gamma$ is stably embedded---any definable set $X
    \subseteq \Gamma^n$ is definable with parameters in $\Gamma$.
  \item \label{gi4} $\Gamma$ eliminates imaginaries---if $X \subseteq
    \Gamma^n$ is definable, then the code $\ulcorner X \urcorner$ is
    interdefinable with a tuple in $\Gamma$.
  \end{enumerate}
\end{fact}
Part (\ref{gi1}) holds because it holds in ACVF (see
\cite[Lemma~1.3.1(i)]{survey}).  Part (\ref{gi2}) is a consequence of
$C$-minimality.  For example, it follows easily from Lemma~\ref{control}
below.  Part (\ref{gi3}) follows from part (\ref{gi2}) by Hasson and
Onshuus's theorem that o-minimal sets are always stably embedded
\cite{embedded-o}.  Part (\ref{gi4}) follows from stable embeddedness
plus the well-known fact that o-minimal expansions of DOAG eliminate
imaginaries if at least one positive element is 0-definable, in this
case $v(t)$.
\begin{lemma} \label{control}
  Let $D \subseteq \Mm$ be definable.  Then there are two subsets $D^{\err},
  \tilde{D} \subseteq \Gamma$ such that the following things hold:
  \begin{itemize}
  \item $\tilde{D}$ is quantifier-free definable in $(\Gamma,+,\le)$,
    i.e., it is a finite union of points and intervals.
  \item $D^{\err}$ is finite.
  \item If $x \in \Mm^\times$ and $v(x) \notin D^{\err}$, then $x \in
    D \iff v(x) \in \tilde{D}$.
  \end{itemize}
\end{lemma}
\begin{proof}
  Say that a set $D \subseteq \Mm$ is ``good'' if it satisfies the
  conclusion of the lemma.  The class of good sets is closed under
  finite boolean combinations.  For example, if $\star$ is $\cap$ or
  $\cup$ (or any other 2-ary boolean operation), then we can take
  \begin{gather*}
    \widetilde{X \star Y} = \tilde{X} \star \tilde{Y} \\
    (X \star Y)^{\err} = X^{\err} \cup Y^{\err}.
  \end{gather*}
  Open balls, closed balls, and points are good by inspection.  By
  $C$-minimality, every definable set $D \subseteq \Mm$ is good.
\end{proof}
This implies o-minimality of $\Gamma$.  Indeed, let $X \subseteq
\Gamma$ be definable, and take $D = v^{-1}(X) = \{x \in \Mm^\times :
v(x) \in X\}$.  If $\tilde{D}$ is a quantifier-free definable set as
in the lemma, then
\begin{equation*}
  v(x) \in \tilde{D} \iff x \in D \iff v(x) \in X
\end{equation*}
except at finitely many values of $v(x)$.  Thus $\tilde{D}$ and $X$
differ at only finitely many points, and $X$ is quantifier-free
definable.

Aside from implying o-minimality of $\Gamma$, Lemma~\ref{control} also
gives a ``domination'' statement similar to compact domination
\cite[Definition~9.1]{goo} or stable domination \cite{HHM2}.
Specifically, if $D \subseteq \Mm$ is definable, then for all but
finitely many $\gamma \in \Gamma$, the annulus $v^{-1}(\gamma)$ lies
entirely inside or outside $D$.

There is an entirely analogous picture for the residue field $k$.  The
proofs for $k$ are similar to the proofs for $\Gamma$.
\begin{fact} \phantomsection \label{k-induce}
  \begin{enumerate}
  \item \label{ki1} The residue field $k$ is an algebraically closed
    field.
  \item \label{ki2} The induced structure on $k$ is strongly minimal.
  \item \label{ki3} $k$ is stably embedded---any definable set $X
    \subseteq k^n$ is definable with parameters in $k$.
  \item \label{ki4} $k$ eliminates imaginaries---if $D \subseteq
    k^n$ is definable, then the code $\ulcorner D \urcorner$ is
    interdefinable with a tuple in $k$.
  \end{enumerate}
\end{fact}
Let $\Oo$ be the valuation ring and $\mm$ be its maximal ideal, so
that $k = \Oo/\mm$.
\begin{lemma} \label{control2}
  If $D \subseteq \Oo$ is definable, then there are subsets
  $\tilde{D}, D^{\err} \subseteq k$ such that
  \begin{enumerate}
  \item $\tilde{D}$ is finite or cofinite in $k$.
  \item $D^{\err}$ is finite.
  \item If $x \in \Oo$ and $\res(x) \notin D^{\err}$, then $x \in D
    \iff \res(x) \in \tilde{D}$.
  \end{enumerate}
\end{lemma}
We are also interested in the interactions between $k$ and $\Gamma$.
\begin{definition} \label{def-ortho}
  Two definable sets $X, Y$ are \emph{orthogonal} if the boolean
  algebra of definable subsets of $X^n \times Y^m$ is generated by
  sets of the form $A \times B$ with $A \subseteq X^n$ and $B
  \subseteq Y^m$.
\end{definition}
If $D \subseteq X^n \times Y^m$ is definable and $a \in X^n$, let $D_a
= \{y \in Y^m : (x,y) \in D\}$.  Then $X$ and $Y$ are orthogonal if
and only if the following holds:
\begin{quote}
  For any definable set $D \subseteq X^n \times Y^m$, the family of
  slices $\{D_a : a \in X^n\}$ is finite.
\end{quote}
Indeed, if $D$ is a finite boolean combination of $\{A_1 \times B_1,
\ldots, A_n \times B_n\}$, then each slice $D_a$ belongs to the finite
boolean algebra generated by $\{B_1,\ldots,B_n\}$.  Conversely, if
$\{D_a : a \in X^n\}$ is a finite set $\{B_1,\ldots,B_n\}$, and we let
$A_i = \{a \in X^n : D_a = B_i\}$, then
\begin{equation*}
  D = \bigcup_{i = 1}^n (A_i \times B_i).
\end{equation*}
\begin{proposition} \phantomsection \label{prop-ortho}
  \begin{enumerate}
  \item If $D \subseteq k^n$ is definable and $f : D \to \Gamma^m$ is
    definable, then $f$ has finite image.
  \item \label{po2} $k$ and $\Gamma$ are orthogonal.
  \item \label{po3} If $D \subseteq \Gamma^m$ is definable and $f : D
    \to k^n$ is definable, then $f$ has finite image.
  \end{enumerate}
\end{proposition}
Again, this is well-known, but we include a proof for completeness.
\begin{proof}
  \begin{enumerate}
  \item The set $D$ and its image $\im(f)$ are stable and stably
    embedded.  On the other hand, $\Gamma^m$ and its subset $\im(f)$
    admit a definable linear ordering.  A stable set admitting a
    definable linear ordering must be finite.
  \item Let $D \subseteq k^n \times \Gamma^m$ be definable.  It
    suffices to show that there are only finitely many possibilities
    for $D_a$ as $a$ ranges over $k^n$.  By the strong elimination of
    imaginaries in $\Gamma$, we can assign a code $\ulcorner D_a
    \urcorner \in \Gamma^\ell$ for each $a \in k^n$ in a uniform,
    definable way.  Then the map $f(a) = \ulcorner D_a \urcorner$ is
    definable, and has finite image by the previous point.  Therefore
    $\{D_a : a \in k^n\}$ is finite as desired.
  \item Apply the previous point to the graph of $\Gamma$. \qedhere
  \end{enumerate}
\end{proof}
The first part of Proposition~\ref{prop-ortho} follows formally from
the fact that $k$ is stable and $\Gamma$ is linearly ordered.  In
contrast, the other two points need the stronger fact that $\Gamma$ is
o-minimal.  For example, contrast the above situation with the
following example:
\begin{example}
  Let $(M,\le,\approx)$ be a generic set with a linear order and
  unrelated equivalence relation.  Let $k$ be the quotient
  $M/{\approx}$.  Then it is not hard to see that $k$ is strongly
  minimal and stably embedded.  However, the quotient map $M \to k$
  shows that the linearly ordered set $M$ and the strongly minimal set
  $k$ are not orthogonal in the sense above.
\end{example}
The home sort $\Mm$ and the value group $\Gamma$ are not orthogonal,
obviously.  Nevertheless, there are some limitations on maps from
$\Mm$ to $\Gamma$:
\begin{fact} \label{const-to-gamma}
  If $D \subseteq \Mm$ and $f : D \to \Gamma^n$ is definable, then $f$
  is locally constant at all but finitely many points of $D$.
\end{fact}
Cubides Kovacsics and Delon \cite[Theorem 2.2(2)]{CK-D} attribute
Fact~\ref{const-to-gamma} to Haskell and Macpherson
\cite[Proposition~3.9?]{cminfields}.

Because $\Gamma$ is stably embedded with elimination of imaginaries,
Fact~\ref{const-to-gamma} implies the following:
\begin{corollary} \label{corfam}
  Let $X \subseteq \Mm$ be definable and let $\{Y_a\}_{a \in X}$ be a
  definable family of subsets of $\Gamma^n$.  Then the family
  $\{Y_a\}$ is locally constant at all but finitely many points of
  $X$.  That is, for all $a$ in a cofinite subset of $X$, there is a
  neighborhood $N \ni a$ such that the family $\{Y_a\}_{a \in X \cap
    N}$ is constant.
\end{corollary}
Using this, we can quickly re-prove some facts from \cite{CK-D}.  This
is not a pointless exercise---we will need the proofs later.
\begin{theorem}[{$\approx$ \cite[Theorem~3.1]{CK-D}}] \label{copy-1}
  If $f : \Mm^\times \to \Mm$ is definable, then there is a definable
  function $g : \Gamma \to \Gamma_\infty$ such that
  \begin{equation*}
    v(f(x)) = g(v(x))
  \end{equation*}
  for all $x$ in a punctured neighborhood of 0.
\end{theorem}
Here, $\Gamma_\infty$ denotes $\Gamma \cup \{+\infty\}$.
\begin{proof}
  Let $D = \{(v(x),v(f(x))) : x \in \Mm^\times\} \subseteq
  \Gamma_\infty^2$.  Then $D$ is the image of $\Mm$ under a definable
  map, and so
  \begin{equation*}
    \dim(D) = \dpr(D) \le \dpr(\Mm) = 1.
  \end{equation*}
  By cell decomposition, $D$ is a disjoint union $\coprod_{i = 1}^n
  C_i$ where the $C_i$ are cells of dimension $\le 1$.  In particular,
  each cell is either a point, a vertical line segment, or the graph
  of a continuous definable function on an interval.  Let $X_i$ be the
  preimage of $C_i$:
  \begin{equation*}
    X_i = \{x \in \Mm^\times : (v(x),v(f(x))) \in C_i\}.
  \end{equation*}
  Then the $X_i$ are a partition of $\Mm^\times$.  By
  Corollary~\ref{inf-cor}, one of the $X_i$ contains a punctured ball
  $B \setminus \{0\}$.  As $x$ ranges over $B \setminus \{0\}$, the
  valuation $v(x)$ takes arbitrarily high values.  Therefore, the
  projection $\pi_1(C_i) \subseteq \Gamma$ has no upper bound.  By definition of ``cell,'' $C_i$ must be the graph of a continuous
  function $g : (a,+\infty) \to \Gamma$ for some $a$.  Then
  \begin{equation*}
    x \in B \setminus \{0\} \implies (v(x),v(f(x))) \in C_i \implies
    v(f(x)) = g(v(x)). \qedhere
  \end{equation*}
\end{proof}
\begin{theorem}[{$\approx$ \cite[Theorem~3.3]{CK-D}}] \label{copy-2}
  If $D \subseteq \Mm$ is definable and infinite and $f : D \to \Mm$
  is definable, then there is a ball $B \subseteq D$ and a definable
  function $g : \Gamma \to \Gamma_\infty$ such that
  \begin{equation*}
    v(f(x)-f(y)) = g(v(x-y)) \text{ for distinct } x,y \in B.
  \end{equation*}
\end{theorem}
\begin{proof}
  For each $a \in D$, let
  \begin{equation*}
    U_a = \{(v(x-a),v(f(x)-f(a))) : x \in D \setminus \{a\}\}.
  \end{equation*}
  By Corollary~\ref{corfam}, there is some point $a_0$ such that $U_a$
  is constant on an open ball $B \ni a_0$:
  \begin{equation*}
    a \in B \implies U_a = U_{a_0}.
  \end{equation*}
  By Theorem~\ref{copy-1}, there is an open ball $B' \ni a_0$ such
  that
  \begin{equation*}
    x \in B' \setminus \{a_0\} \implies v(f(x)-f(a_0)) = g(v(x-a_0)). \tag{$\ast$}
  \end{equation*}
  Replacing $B$ and $B'$ with their intersection, we may assume $B =
  B'$.  If $\gamma$ is the valuative radius of $B'$, then
  \begin{align*}
    &U_{a_0} \cap ((\gamma,+\infty) \times \Gamma) = \{(x,g(x)) : x > \gamma\} \text{ by ($\ast$),}\\
    \text{ and so } &U_a \cap ((\gamma,+\infty) \times \Gamma) = \{(x,g(x)) : x > \gamma\} \text{ for } a \in B. \tag{$\dag$}
  \end{align*}
  If $x,a \in B$, then $(v(x-a),v(f(x)-f(a))) \in U_a$ and $v(x-a) >
  \gamma$, so ($\dag$) shows
  \begin{equation*}
    v(f(x)-f(a)) = g(v(x-a)). \qedhere
  \end{equation*}
\end{proof}

\subsection{Unary sets in RV}
Recall that $\RV$ is the group $\Mm^\times/(1 + \mm)$, written
multiplicatively.  If $a \in \Mm^\times$, then $\rv(a)$ denotes the
image in $\RV$.  Note that for $a,b \in \Mm^\times$, we have
\begin{equation*}
  \rv(a) = \rv(b) \iff v(a-b) > v(b).
\end{equation*}
The group $\RV$ sits in a short exact
sequence
\begin{equation*}
  1 \to \kx \to \RV \to \Gamma \to 1.
\end{equation*}
If $X \subseteq \Gamma$, let $\RV_X$ denote the preimage of $X$ under
the map $\RV \to \Gamma$.  For $\gamma \in \Gamma$, we write
$\RV_\gamma$ for $\RV_{\{\gamma\}}$.  For example, $\RV_0$ is the
subgroup $\kx \subseteq \RV$.  In general, the sets $\RV_\gamma$
are the cosets of $\kx$ in $\RV$.  In particular, each set
$\RV_\gamma$ is strongly minimal, and orthogonal to $\Gamma$ in the
sense of Definition~\ref{def-ortho}.
\begin{lemma} \label{first-forbidden}
  Let $D \subseteq \RV$ be definable.  Then
  \begin{equation*}
    \RV_\gamma \subseteq D \text{ or } \RV_\gamma \cap D = \varnothing
  \end{equation*}
  for all but finitely many $\gamma \in \Gamma$.
\end{lemma}
\begin{proof}
  Let $X = \{x \in \Mm^\times : \rv(x) \in D\}$.  Applying
  Lemma~\ref{control} to $X$ gives two definable sets $\tilde{X},
  X^\err \subseteq \Gamma$, with $X^\err$ finite, such that
  \begin{equation*}
    v(x) \notin X^\err \implies (x \in X \iff v(x) \in \tilde{X}).
  \end{equation*}
  Fix some $\gamma \in \Gamma \setminus X^\err$.  If $\rv(x) \in
  \RV_\gamma$, then $v(x) = \gamma \notin X^\err$, and so
  \begin{equation*}
    \rv(x) \in D \iff x \in X \iff v(x) \in \tilde{X} \iff \gamma \in
    \tilde{X}.
  \end{equation*}
  Therefore
  \begin{gather*}
    \gamma \in \tilde{X} \implies \RV_\gamma \subseteq D \\
    \gamma \notin \tilde{X} \implies \RV_\gamma \cap D = \varnothing. \qedhere
  \end{gather*}
\end{proof}
\begin{corollary} \label{proto-bad}
  Suppose $D \subseteq \RV$ is definable, and $D \cap \RV_\gamma$ is
  finite for every $\gamma$.  Then $D$ is finite.
\end{corollary}
\begin{proof}
  Otherwise, there are infinitely many $\gamma$ such that $D$ contains
  some but not all of $\RV_\gamma$.
\end{proof}
\begin{theorem} \label{rv-unary}
  Let $D \subseteq \RV$ be definable.  Then there is a definable set
  $\tilde{D} \subseteq \Gamma$ such that $D$ differs from
  $\RV_{\tilde{D}}$ at only finitely many points.
\end{theorem}
\begin{proof}
  Let $\tilde{D}$ be the set of points $\gamma \in \Gamma$ such that
  $D \cap \RV_\gamma$ is infinite.  The set $\tilde{D}$ is definable
  because $\RV_\gamma$ admits a definable bijection\footnote{\ldots of
  bounded complexity as $\gamma$ varies} to the strongly minimal set
  $\kx$, and strongly minimal sets eliminate $\exists^\infty$.
  By strong minimality of $\RV_\gamma$,
  \begin{equation*}
    \tilde{D} = \{\gamma \in \Gamma : D \cap \RV_\gamma \text{ is
      cofinite in } \RV_\gamma\}.
  \end{equation*}
  Let $S$ be the symmetric difference of $D$ and $\RV_{\tilde{D}}$.
  If $\gamma \in \tilde{D}$, then $D$ and $\RV_{\tilde{D}}$ both
  contain almost all of $\RV_\gamma$.  If $\gamma \notin \tilde{D}$,
  then $D$ and $\RV_{\tilde{D}}$ both contain almost none of
  $\RV_\gamma$.  Either way, the symmetric difference $S \cap
  \RV_\gamma$ is finite by strong minimality of $\RV_\gamma$.  Then
  $S$ is finite by Corollary~\ref{proto-bad}.
\end{proof}

\subsection{Linear functions on $\Gamma$}
Recall that the value group $\Gamma$ is an o-minimal expansion of
DOAG, the theory of divisible ordered abelian groups.  In the
following section, we work in $\Gamma$, but the arguments apply
equally well to any o-minimal expansion of DOAG.
\begin{definition} \label{vexity}
  Let $I$ be an interval in $\Gamma$.  A definable function $f : I \to
  \Gamma$ is \emph{linear} if
  \begin{equation*}
    f(a) + f(a+\epsilon_1+\epsilon_2) = f(a + \epsilon_1) +
    f(a+\epsilon_2)
  \end{equation*}
  for any $a, \epsilon_1, \epsilon_2$ such that
  $\{a,a+\epsilon_1,a+\epsilon_2,a+\epsilon_1+\epsilon_2\} \subseteq
  I$.
\end{definition}
If $f : (a,b) \to \Gamma$ is a function and $\epsilon > 0$, define a
function
\begin{gather*}
  \delta_\epsilon f : (a,b-\epsilon) \to \Gamma \\
  \delta_\epsilon f (x) = f(x+\epsilon) - f(x).
\end{gather*}
Note that $\delta_{\epsilon_1} \delta_{\epsilon_2} f =
\delta_{\epsilon_2} \delta_{\epsilon_1} f$.
\begin{remark} \label{increasing-duh}
  If $f : (a,b) \to \Gamma$ is definable, then the following are
  equivalent:
  \begin{enumerate}
  \item \label{iduh1} $f$ is constant.
  \item \label{iduh2} $f$ is locally constant, in the sense that for
    every $x \in (a,b)$ there is an interval $I$ around $x$ such that
    $f \restriction I$ is constant.
  \item \label{iduh3} $\delta_\epsilon f = 0$ for every $\epsilon > 0$.
  \item \label{iduh4} There is $\epsilon_0$ such that $\delta_\epsilon
    f = 0$ for $\epsilon \in (0,\epsilon_0)$.
  \end{enumerate}
  Indeed, the following implications are clear
  \begin{equation*}
    (\ref{iduh1})\iff(\ref{iduh3})\implies(\ref{iduh4})\implies(\ref{iduh2}),
  \end{equation*}
  and the implication (\ref{iduh2})$\implies$(\ref{iduh1}) is a
  well-known consequence of o-minimality, part of the monotonicity
  theorem \cite[Theorem 3.1.2]{lou-o-minimality}.
\end{remark}
\begin{lemma} \label{seven}
  Let $f : (a,b) \to \Gamma$ be definable.  The following seven
  conditions are equivalent:
  \begin{enumerate}
  \item \label{s1} $f$ is linear.
  \item \label{s2} $\forall \epsilon_1 > 0 ~ \forall \epsilon_2 > 0 :
    \delta_{\epsilon_2} \delta_{\epsilon_1} f = 0$.
  \item \label{s3} $\forall \epsilon_1 > 0 : (\delta_{\epsilon_1} f
    \text{ is a constant function})$.
  \item \label{s4} $\forall \epsilon_1 > 0 ~ \exists \epsilon_0 > 0 ~
    \forall \epsilon_2 \in (0,\epsilon_0) : \delta_{\epsilon_2}
    \delta_{\epsilon_1} f = 0$.
  \item \label{s5} $\exists \epsilon_0 > 0 ~ \forall \epsilon_1 > 0 ~
    \forall \epsilon_2 \in (0,\epsilon_0) : \delta_{\epsilon_2}
    \delta_{\epsilon_1} f = 0$
  \item \label{s6} $\exists \epsilon_0 > 0 ~ \forall \epsilon_2 \in
    (0,\epsilon_0) ~ \forall \epsilon_1 > 0 : \delta_{\epsilon_1}
    \delta_{\epsilon_2} f = 0$.
  \item \label{s7} $\exists \epsilon_0 > 0 ~ \forall \epsilon_2 \in
    (0,\epsilon_0) : (\delta_{\epsilon_2} f \text{ is a constant function})$.
  \end{enumerate}
\end{lemma}
\begin{proof}
  (\ref{s1})$\iff$(\ref{s2}) by unwinding the definitions.  The
  equivalence of (\ref{s2}), (\ref{s3}), and (\ref{s4}) comes from
  applying Remark~\ref{increasing-duh} to the function
  $\delta_{\epsilon_1} f$.  The implications
  \begin{gather*}
    (\ref{s2})\implies(\ref{s5})\implies(\ref{s4}) \\
    (\ref{s5})\iff(\ref{s6})
  \end{gather*}
  are clear.  Finally, (\ref{s6})$\iff$(\ref{s7}) by applying
  Remark~\ref{increasing-duh} to the function $\delta_{\epsilon_2} f$.
\end{proof}
Focusing on conditions (\ref{s1}), (\ref{s3}), and (\ref{s7}), we see
\begin{corollary} \label{reduce-seven}
  Let $f : (a,b) \to \Gamma$ be definable.  The following are
  equivalent:
  \begin{enumerate}
  \item $f$ is linear.
  \item $\delta_\epsilon f$ is a constant function for every $\epsilon
    > 0$.
  \item $\delta_\epsilon f$ is a constant function, for all
    sufficiently small $\epsilon > 0$.
  \end{enumerate}
\end{corollary}
\begin{lemma} \label{convex-merge}
  Suppose $a < b < c < d$ and $f : (a,d) \to \Gamma$ is such that the
  two restrictions $f \restriction (a,c)$ and $f \restriction (b,d)$
  are linear.  Then $f$ is linear.
\end{lemma}
\begin{proof}
  Suppose $0 < \epsilon < c-b$.  By Corollary~\ref{reduce-seven},
  $\delta_\epsilon f$ is constant on $(a,c-\epsilon)$ and
  on $(b,d-\epsilon)$.  These intervals overlap, by choice of
  $\epsilon$, and so $\delta_\epsilon f$ is constant on
  $(a,d-\epsilon)$.  By Corollary~\ref{reduce-seven} again, $f$ is
  linear.
\end{proof}
\begin{theorem} \label{local-global-convex}
  Let $f : (a,b) \to \Gamma$ be definable.  Suppose that for every $x
  \in (a,b)$, there is an interval $I \ni x$ such that $f \restriction
  I$ is linear.  Then $f$ is linear.
\end{theorem}
\begin{proof}
  Take some small interval $(a_0,b_0) \subseteq (a,b)$ such that $f
  \restriction (a_0,b_0)$ is linear.  Let
  \begin{gather*}
    S = \{x \in (a_0,b] : f \text{ is linear on } (a_0,x)\} \\
    b_1 = \sup S.
  \end{gather*}
  The set $S$ is definable, so $b_1$ exists by o-minimality.  The open
  interval $(a_0,b_1)$ is a union of an increasing chain of intervals on
  which $f$ is linear.  Therefore, $f$ is linear on
  $(a_0,b_1)$, and so $b_1 \in S$, and $b_1 = \max S$.

  We claim that $b_1 = b$.  Otherwise, $b_1 \in (a,b)$ and there is a
  small interval $(c,d)$ around $b_1$ on which $f$ is linear.  By
  Lemma~\ref{convex-merge}, $f$ is linear on the union
  $(a_0,b_1) \cup (c,d) \supseteq (a_0,d)$.  But then $d \in S$ and $d
  > b_1$, contradicting the fact that $b_1 = \sup S$.

  Then $b = b_1 = \max S \in S$, meaning that $f$ is linear
  on $(a_0,b)$.  A similar argument shows that $f$ is linear
  on $(a,b_0)$.  By Lemma~\ref{convex-merge}, $f$ is linear
  on the union $(a,b_0) \cup (a_0,b) = (a,b)$.
\end{proof}

\section{A spiritual overview of the induced structure on $\RV$} \label{sec-vague}
This section is a vague description of the intuition I have for
Section~\ref{sec-heart}.  It is provided solely for motivation, to
help the reader follow the next section.  Nothing I say here will be
used in the later proofs, and the reader who doesn't care about
motivation can skip straight to Section~\ref{sec-heart}.  In fact,
everything I say here should be treated as unofficial speculation, not
taken seriously.

\subsection{A first approximation}
To a first approximation, we can summarize the situation as follows:
\begin{slogan} \label{slogan-1}
  All the induced structure on $\RV$ is generated by the induced
  structure on $k$ and the induced structure on $\Gamma$.
\end{slogan}
In a $C$-minimal field, the induced structure on $k$ can be an arbitrary
strongly minimal expansion of ACF, and the induced structure on
$\Gamma$ can be an arbitrary o-minimal expansion of DOAG.
Slogan~\ref{slogan-1} is saying that the additional structure on RV
solely comes from the additional structure on $k$ and the induced
structure on $\Gamma$.

As a consequence, definable sets in $\RV^n$ are built out of the
following components:
\begin{itemize}
\item Definable sets on $\Gamma^n$, pulled back to $\RV^n$.
\item Sets definable in the group structure of $\RV$, such as
  \begin{equation*}
    \{(y,x) \in \RV : y^m = x^m\}. \tag{$\ast$}
  \end{equation*}
\item Definable sets in $k^n$, pushed forward to $\RV^n$ and
  translated.
\end{itemize}
If $D \subseteq \RV^2$ has dp-rank 1, then $D$ cannot contain the
preimage of something pulled back from $\RV^2 \to \Gamma^2$, because
the fibers of $\RV^2 \to \Gamma^2$ have dp-rank 2.  Consequently, $D$
must be built out of translates of definable sets in $k^2$, which push
forward to points in $\Gamma^2$, and sets like ($\ast$), which push
forward to line segments in $\Gamma^2$ of rational slope.  So one concludes
that
\begin{quote}
  If $D \subseteq \RV^2$ has dp-rank 1, then the image of $D$ in
  $\Gamma^2$ is made of points and line segments of rational slope.
\end{quote}
However, Slogan~\ref{slogan-1} is not quite right.

\subsection{A second approximation}
Let $E_\Gamma$ be the ring of definable endomorphisms of $(\Gamma,+)$.
Then $E_\Gamma$ is a skew field (possibly non-commutative).  Let
$E^0_\kx$ be the ring of definable endomorphisms of $\kx$.  Unlike
$E_\Gamma$, $E^0_\kx$ is not a skew field.  For example, if $\kx$ is a
pure ACF of characteristic prime to $n > 1$, then the endomorphism
$f(x) = x^n$ has no inverse.  On the other hand, $E^0_\kx$ is
commutative: the action of $E^0_\kx$ on the torsion in $\kx$ is
faithful (by strong minimality) giving an embedding of $E^0_\kx$ into
the ring of endomorphisms of the abelian group $\Qq/\Zz$.  But
$\End(\Qq/\Zz)$ is the profinite completion $\hat{\Zz}$ of $\Zz$,
which is a commutative ring.  So the embedding $E^0_\kx
\hookrightarrow \hat{\Zz}$ shows that $E^0_\kx$ is commutative.  This
also shows that $E^0_\kx$ is small, unlike $E_\Gamma$.

Let $E_\kx$ be the ring obtained from $E^0_\kx$ by inverting the
natural numbers.  Elements of $E_\kx$ are the definable ``endogenies''
of $\kx$, in the sense of \cite[Proof of Theorem~1.14]{Poizat}.  Now
$E_\kx$ is a field (not just a skew field).  When $\Gamma$ and $\kx$
have the structure of a pure DOAG and pure ACF, the two skew fields
$E_\Gamma$ and $E_\kx$ are both $\Qq$.

\begin{slogan} \label{slogan-2}
  There is an analogous ring $E_{\RV}$ of definable endogenies of $\RV$, and $E_{\RV}$ embeds into both $E_\kx$ and $E_\Gamma$:
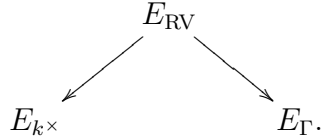
    \begin{equation*}
      \xymatrix{ & E_{\RV} \ar[dl] \ar[dr] & \\ E_\kx & & E_\Gamma.}
    \end{equation*}
    In particular, $E_{\RV}$ is a small commutative field of characteristic 0.

    Finally, the induced structure on $\RV$ is generated by the
    induced structure on $\kx$, the induced structure on $\Gamma$, and
    the action of $E_{\RV}$ on $\RV$.
\end{slogan}
Compared to Slogan~\ref{slogan-1}, we now have to deal with definable
sets like
\begin{equation*}
  \{(x,y) \in \RV^2 : y = x^q\} \text{ for } q \in E_{\RV}.
\end{equation*}
This leads to the following consequence for sets of dp-rank 1:
\begin{quote}
  If $D \subseteq \RV^2$ has dp-rank 1, then the image of $D$ in
  $\Gamma^2$ is made of points and line segments having slopes in $E_{\RV}$.
\end{quote}
More precisely, the slopes of the lines live in the image of $E_{\RV}
\to E_\Gamma$.

Note that the field $E_{\RV}$ ties together $E_\kx$ and $E_\Gamma$ in
a strange way.  In order for there to be strange new sets in $\RV$, we
need $E_{\RV} \supsetneq \Qq$.  This requires both $E_\kx$ and
$E_\Gamma$ to be bigger than $\Qq$, so \emph{both} $\kx$ and $\Gamma$
need ``exotic automorphisms.''

Slogan~\ref{slogan-2} is still missing part of the picture.
\subsection{A third approximation}
Let $(a,b)$ and $(c,d)$ be two intervals in $\Gamma$ containing $0$.
Let $f : (a,b) \to \Gamma$ and $g : (c,d) \to \Gamma$ be definable
linear functions with $f(0) = g(0) = 0$.  Say that $f$ and $g$ have
the same ``slope'' if $f = g$ on a neighborhood of 0, or equivalently,
on the intersection $(a,b) \cap (c,d)$.  The set of slopes is
naturally a skew field $E^+_\Gamma$ extending $E_\Gamma$.  An element
of $E^+_\Gamma$ is something like an endomorphism of $\Gamma$ defined
only locally around 0.  If I understand correctly, there are
well-known o-minimal theories in which $E^+_\Gamma \supsetneq
E_\Gamma$.
\begin{slogan} \label{slogan-3}
  There is an analogous skew field $E^+_\RV$ for $\RV$, fitting into
  the following diagram of embeddings of skew fields:
  
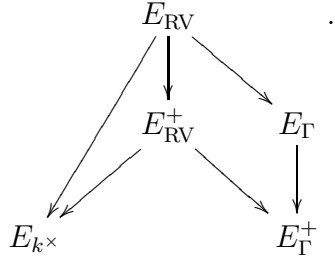
\begin{equation*}
    \xymatrix{ & E_{\RV} \ar[ddl] \ar[dr] \ar[d] & \\ & E^+_{\RV} \ar[dl] \ar[dr] & E_\Gamma \ar[d] \\ E_\kx & & E^+_\Gamma}.
  \end{equation*}
  In particular, $E^+_{\RV}$ is a small commutative field of characteristic 0.

  Finally, the induced structure on $\RV$ is generated by the induced
  structure on $\kx$, the induced structure on $\Gamma$, and the
  action of $E^+_{\RV}$ on $\RV$.
\end{slogan}
The point is that we have to deal with line segments, not just lines.
This leads to the following consequence for sets of dp-rank 1:
\begin{quote}
  If $D \subseteq \RV^2$ has dp-rank 1, then the image of $D$ in
  $\Gamma^2$ is made of points and line segments having slopes in
  $E_{\RV}^+$.
\end{quote}
We now return to the real proof.

\section{Definable sets on $\RV \times \RV$} \label{sec-heart}
Recall that $\RV_a$ is the fiber of $\RV \to \Gamma$ over $a \in
\Gamma$.  Then $\RV_a \cdot \RV_b = \RV_{a+b}$, and $\RV_0 = \kx
\subseteq \RV$.  The sets $\RV_a$ are the cosets of $\kx$ in $\RV$.
In particular, each $\RV_a$ is in definable bijection with $\kx$.

Because $\kx$ is strongly minimal, there is a good notion of
dimension for definable subsets of $(\kx)^n$.  Dimension is
definable in families, and agrees with Morley rank and dp-rank.
The set $\RV_a \times \RV_b$ is in definable bijection
with $(\kx)^2$, so there is also a good
dimension theory on $\RV_a \times \RV_b$.  Note that if $D \subseteq
\RV_a \times \RV_b$, then
\begin{gather*}
  \dim(D) \in \{-\infty,0,1,2\} \\
  \dim(D) = -\infty \iff D = \varnothing \\
  \dim(D) \le 0 \iff |D| < \infty \\
  \dim(D) = 2 \iff \dim((\RV_a \times \RV_b) \setminus D) \ne 2, \tag{$\ast$}
\end{gather*}
because the analogous properties hold in $\kx \times \kx$.
Condition ($\ast$) has to do with the fact that $\kx \times
\kx$ and $\RV_a \times \RV_b$ have Morley degree 1.
\begin{definition}
  Let $D \subseteq \RV^2$ be definable.  For $(a,b) \in \Gamma^2$, let
  \begin{equation*}
    D_{a,b} = D \cap (\RV_a \times \RV_b).
  \end{equation*}
  Let $\tilde{D}$ and $\overline{D}$ be the following subsets of
  $\Gamma^2$:
  \begin{align*}
    \tilde{D} &= \{(a,b) \in \Gamma^2 : \dim(D_{a,b}) = 2\} \\
    \overline{D} &= \{(a,b) \in \Gamma^2 : D_{a,b} \ne \varnothing\} \\
    &= \{(a,b) \in \Gamma^2 : \dim(D_{a,b}) \ge 0\}.
  \end{align*}
\end{definition}
Then $\overline{D}$ is the image of $D$ under the natural map $\RV^2 \to
\Gamma^2$, and $\tilde{D}$ is the set of points $(a,b) \in \Gamma^2$
such that $D$ contains ``most'' points of the fiber $\RV_a \times
\RV_b$.  The sets $\overline{D}$ and $\tilde{D}$ are definable sets,
because dimension is definable in $\kx$.  Note that $D \mapsto
\tilde{D}$ is a homomorphism of boolean algebras, by ($\ast$) above.
For example, if $D$ and $E$ are complementary subsets of $\RV^2$, then
$D_{a,b}$ and $E_{a,b}$ are complementary subsets of $\RV_a \times
\RV_b$, exactly one of the two has dimension 2, and so $\tilde{D}$ is
complementary to $\tilde{E}$.
\begin{remark} \label{small-analysis}
  Let $v : \RV \to \Gamma$ denote the natural map.  Let $D \subseteq
  \RV^2$ be definable.  Then $D$ can be decomposed as a symmetric
  difference
  \begin{equation*}
    D = D' \mathbin{\Delta} X,
  \end{equation*}
  where $D' = \{(x,y) \in \RV^2 : (v(x),v(y)) \in \tilde{D}\}$ and
  \begin{equation*}
    \tilde{X} = \widetilde{D \mathbin{\Delta} D'} = \tilde{D}
    \mathbin{\Delta} \widetilde{D'} = \tilde{D} \mathbin{\Delta}
    \tilde{D} = \varnothing.
  \end{equation*}
\end{remark}
\begin{definition}
  A definable set $D \subseteq \RV^2$ is \emph{small over $\Gamma^2$}
  if $\tilde{D} = \varnothing$, meaning that
  \begin{equation*}
    \dim(D_{a,b}) = \dim(D \cap (\RV_a \times \RV_b)) < 2
  \end{equation*}
  for every $(a,b) \in \Gamma^2$.
\end{definition}
Remark~\ref{small-analysis} says that every definable set in $\RV$ is
a symmetric difference of a definable set pulled back from $\Gamma^2$
and a definable set that is small over $\Gamma^2$.  In light of this,
we will focus on sets that are small over $\Gamma^2$.
\begin{definition}
  A \emph{fragment} is a non-empty definable set $X \subseteq \RV_a
  \times \RV_b$ for some $(a,b) \in \Gamma^2$, such that $\dim(X) <
  2$.  We say that $X$ is a \emph{central fragment} if $a=b=0$.
\end{definition}
Recall that $\overline{D}$ is the image of $D$ under $\RV^2 \to \Gamma^2$,
or equivalently, the set of $(a,b) \in \Gamma^2$ such that $D_{a,b}
\ne \varnothing$.  If $D$ is small over $\Gamma^2$, then $D_{a,b}$ is
a fragment for every $(a,b) \in \overline{D}$.

Say that two fragments $X,Y$ are \emph{equivalent} if one is a
translate of the other: $Y = (u,v) \cdot X$ for some $(u,v) \in
\RV^2$.  Every fragment $X$ is equivalent to at least one central
fragment, namely $(u,v)^{-1} \cdot X$ for any $(u,v) \in X$.  Recall
that central fragments are subsets of $\RV_0 \times \RV_0 =
(\kx)^2$.  By elimination of imaginaries in $\kx$, we can
assign a value $\Shape(X)$ to each central fragment $X$ in such a way
that
\begin{enumerate}[(a)]
\item \label{la} $\Shape(X)$ is a finite tuple in $k$.
\item \label{lb} $\Shape(X) = \Shape(Y)$ if and only if $X$ is equivalent to $Y$.
\item \label{lc} $\Shape(X)$ depends definably on $X$, as $X$ varies in a
  definable family.
\end{enumerate}
We call $\Shape(X)$ the \emph{shape} of $X$.  If $X$ is any fragment,
we define $\Shape(X) = \Shape(X')$ for some central fragment $X'$
equivalent to $X$.  The choice of $X'$ doesn't matter, and points
\ref{la}--\ref{lc} continue to hold.  In particular, $\Shape(X)$
is still a finite tuple in $k$, even though $X$ no longer lives in a
power of $k$.

If $X$ is a fragment, let $\Stab(X)$ denote the multiplicative setwise
stabilizer of $X$, i.e., the subgroup
\begin{equation*}
  \Stab(X) = \{(u,v) \in \RV^2 : (u,v) \cdot X = X\}.
\end{equation*}
More generally, if $X,Y$ are two fragments, let
\begin{equation*}
  \Mov(X,Y) = \{(u,v) \in \RV^2 : (u,v) \cdot X = Y\}.
\end{equation*}
The following properties are straightforward to verify.
\begin{observation}  \label{tricks}
  Let $X$ and $Y$ be fragments.
  \begin{enumerate}
  \item $\Mov(X,X) = \Stab(X)$.
  \item $\Shape(X) = \Shape(Y) \iff \Mov(X,Y) \ne \varnothing$.
  \item $\Shape(X)$ determines $\Stab(X)$: if $X$ and $Y$ are
    equivalent, then $\Stab(X) = \Stab(Y)$.
  \item If $X$ and $Y$ are equivalent, then $\Mov(X,Y)$ is a coset of
    the two equal groups $\Stab(X) = \Stab(Y)$.
  \item If $X \subseteq \RV_a \times \RV_b$ and $Y \subseteq \RV_c
    \times \RV_d$, then $\Mov(X,Y) \subseteq \RV_{c-a} \times
    \RV_{b-d}$.  (This uses the fact that $X$ and $Y$ are non-empty.)
  \item In particular, $\Stab(X) \subseteq \RV_0 \times \RV_0 =
    (\kx)^2$.
  \end{enumerate}
\end{observation}
If $s = \Shape(X)$, define $\Stab(s) := \Stab(X)$.  This is
well-defined by part (3) of Observation~\ref{tricks}.  By the above,
$\Stab(s)$ is a definable subgroup of $(\kx)^2$.  If $(u,v) \in X$,
then
\begin{equation*}
  \Stab(s) \cdot (u,v) = \Stab(X) \cdot (u,v) \subseteq X,
\end{equation*}
so $\dim(\Stab(s)) \le \dim(X) \le 1$.  The stabilizer $\Stab(s)$
always contains $(1,1)$, so $\Stab(s)$ is non-empty and
$\dim(\Stab(s)) \ge 0$.  Therefore, the dimension of $\Stab(s)$ is
either 0 or 1.
\begin{definition}
  A shape $s$ is \emph{symmetric} if $\Stab(s)$ has dimension 1, and
  \emph{asymmetric} if $\Stab(s)$ has dimension 0.
\end{definition}
So $s$ is symmetric if $\Stab(s)$ is infinite, and asymmetric if
$\Stab(s)$ is finite.
\begin{lemma} \label{symmetric-shape}
  If $\Shape(X)$ is a symmetric shape $s$, then $X$ is a finite union
  of cosets of $\Stab(s)$.
\end{lemma}
\begin{proof}
  Replacing $X$ with a translate, we may assume that $X$ is a central
  fragment ($X \subseteq (\kx)^2$).  Let $G$ be the 1-dimensional
  subgroup $\Stab(s) = \Stab(X) \subseteq (\kx)^2$.  The fact
  that $G$ stabilizes $X$ means that $X$ is the pullback of some
  definable set along the quotient map
  \begin{equation*}
    (\kx)^2 \to (\kx)^2/G.
  \end{equation*}
  Because $(\kx)^2$ is connected, so is the 1-dimensional
  quotient $(\kx)^2/G$, implying that $(\kx)^2/G$ is
  strongly minimal.  Then $X$ is the pullback of some finite or
  cofinite subset of $(\kx)^2/G$.  It follows that $X$ or the
  complement $(\kx)^2 \setminus X$ is a finite union of cosets of
  $G$.  However, if $X$ is the complement of a finite union of cosets,
  then $\dim(X) = 2$, contradicting the assumption that $X$ is a
  fragment.
\end{proof}
Most definable sets in $(\kx)^2$ cannot be written as finite
unions of translates of 1-dimensional subgroups, so ``most'' shapes
are asymmetric.
\begin{lemma} \label{finshapes}
  Let $D \subseteq \RV^2$ be a definable set which is small over
  $\Gamma^2$, in the sense that $\dim(D_{a,b}) < 2$ for every $(a,b)
  \in \Gamma^2$.  Recall that $\overline{D}$ is the image of $D$ under
  $\RV^2 \to \Gamma^2$.
  \begin{enumerate}
  \item The following set is finite:
    \begin{equation*}
      \{\Shape(D_{a,b}) : (a,b) \in \overline{D}\}.
    \end{equation*}
  \item There is a cell decomposition $\overline{D} = \coprod_{i = 1}^n
    C_i$ and a list of shapes $s_1,\ldots,s_n$ such that
    $\Shape(D_{a,b}) = s_i$ for $(a,b) \in C_i$.
  \end{enumerate}
\end{lemma}
\begin{proof}
  The map $(a,b) \mapsto \Shape(D_{a,b})$ is a definable map from
  $\overline{D}$ to a power of $k$.  By
  Proposition~\ref{prop-ortho}(\ref{po3}), the image is finite.  Then
  the fibers of this map are a partition of $\overline{D}$ into finitely
  many definable sets.  Take a cell decomposition which respects this
  partition.
\end{proof}
\begin{definition}
  A \emph{bad family} is a definable family $\{S_i\}_{i \in I}$ where
  $I$ is an interval in $\Gamma$ and $S_i$ is a finite non-empty
  subset of $\RV_{f(i)}$ for each $i \in I$, for some continuous,
  definable, non-constant function $f : I \to \Gamma$.
\end{definition}
Given a bad family, we can restrict $I$ to get a bad family for which
$f$ is injective, by the monotonicity theorem for o-minimal
structures.  Then the union $\bigcup_{i \in I} S_i$ is a definable set
containing finitely many points in $\RV_\gamma$ for each $\gamma$ in
the infinite set $f(I)$.  This contradicts Corollary~\ref{proto-bad},
and so
\begin{quote}
  Bad families do not exist.
\end{quote}
Here is an example application:
\begin{lemma} \label{example-lemma}
  Let $D \subseteq \RV^2$ be a definable set and let
  \begin{equation*}
    X = \{(a,b) \in \overline{D} : D_{a,b} \text{ is finite and non-empty}\}.
  \end{equation*}
  Then $X$ is finite.
\end{lemma}
\begin{proof}
  The set $X$ can be described as $\{(a,b) : \dim(D_{a,b}) = 0\}$; it
  is definable because dimension is definable in $k$.  Alternatively,
  definability of $X$ follows from Lemma~\ref{finshapes}, because
  whether $D_{a,b}$ is finite is determined by $\Shape(D_{a,b})$.

  If $X$ is infinite, then by cell decomposition, it contains the
  graph of a continuous function, or the transpose of such a graph.
  In other words, there is a continuous definable function $f : (a,b)
  \to \Gamma$ such that $X$ contains one of the sets
  \begin{gather*}
    \Gamma(f) = \{(x,f(x)) : x \in (a,b)\} \\
    \Gamma(f)^T = \{(f(x),x) : x \in (a,b)\}.
  \end{gather*}
  Suppose $\Gamma(f) \subseteq X$.  (The other case is similar.)  Then
  for every $x \in (a,b)$, the set $D_{x,f(x)}$ is a finite non-empty
  subset of $\RV_x \times \RV_{f(x)}$.  Let $\pi : \RV_x \times
  \RV_{f(x)} \to \RV_x$ be the projection onto the first coordinate.
  Then $\{\pi(D_{x,f(x)})\}_{x \in (a,b)}$ is a bad family.
\end{proof}
Here is a more sophisticated example:
\begin{theorem} \label{asymcase}
  Let $D \subseteq \RV^2$ be a definable set and let $s$ be an
  asymmetric shape.  The set $D^s = \{(x,y) \in \Gamma^2 :
  \Shape(D_{x,y}) = s\}$ is finite.
\end{theorem}
\begin{proof}
  As in the proof of Lemma~\ref{example-lemma}, we may assume that
  $D^s$ contains the graph of a continuous definable function $f :
  (a,b) \to \Gamma$.  Fix some central fragment $X \subseteq (\kx)^2$ with
  $\Shape(X) = s$.  For any $x \in (a,b)$, we have
  \begin{gather*}
    (x,f(x)) \in D^s \\ \Shape(D_{x,f(x)}) = s = \Shape(X).
  \end{gather*}
  By Observation~\ref{tricks},
  \begin{gather*}
    \Mov(X,D_{x,f(x)}) \ne \varnothing \\
    \Mov(X,D_{x,f(x)}) \text{ is a coset of } \Stab(s) \\
    \Mov(X,D_{x,f(x)}) \subseteq \RV_x \times \RV_{f(x)}.
  \end{gather*}
  Since $\Stab(s)$ is finite, we see that
  $\Mov(X,D_{x,f(x)})$ is a finite non-empty subset of
  $\RV_{x} \times \RV_{f(x)}$ for any $x \in (a,b)$.
  Projecting onto $\RV_{x}$, we get a bad family as in
  Lemma~\ref{example-lemma}.
\end{proof}
Intuitively, most shapes are asymmetric, which makes the next
corollary rather surprising:
\begin{corollary} \label{asymcase-cor}
  Let $D \subseteq \RV^2$ be a definable set that is small over
  $\Gamma^2$.  Then $\Shape(D_{x,y})$ is symmetric for all but
  finitely many $(x,y)$ in $\overline{D}$.
\end{corollary}
\begin{proof}
  By Lemma~\ref{finshapes}, the set $\{\Shape(D_{x,y}) : (x,y) \in
  \overline{D}\}$ is finite.  Let $s_1,\ldots,s_n$ enumerate the
  asymmetric shapes in this set.  Then
  \begin{equation*}
    \{(x,y) \in \overline{D} : \Shape(D_{x,y}) \text{ is asymmetric}\}
    = \bigcup_{i=1}^n \{(x,y) \in \overline{D} : \Shape(D_{x,y}) = s_i\}
  \end{equation*}
  and the union on the right is finite by Theorem~\ref{asymcase}.
\end{proof}
Theorem~\ref{asymcase} describes the set $D^s := \{(x,y) \in
\overline{D} : \Shape(D_{x,y}) = s\}$ when $s$ is an asymmetric shape.
What if $s$ is a symmetric shape?  By Lemma~\ref{symmetric-shape},
symmetric shapes are closely connected to 1-dimensional subgroups of
$(\kx)^2$.  If $(n,m)$ are coprime integers, let $Q_{n/m}$ be the
following definable subgroup of $(\kx)^2$:
\begin{equation*}
  Q_{n/m} = \{(x,y) \in (\kx)^2 : y^m = x^n\}.
\end{equation*}
Roughly speaking, this is the graph of the ``map'' $f(x) = x^{n/m}$.
Each group $Q_{n/m}$ is connected---there is a definable isomorphism
from $\kx$ to $Q_{n/m}$ given by $z \mapsto (z^m,z^n)$.
Therefore, a definable subgroup $G \subseteq (\kx)^2$ is
commensurable to $Q_{n/m}$ if and only if $G \supseteq Q_{n/m}$ and
$|G/Q_{n/m}|$ is finite.

In the pure language of rings $\Lr$, the only definable automorphisms
of $\kx$ are $x^{\pm 1}$ if $\characteristic(k) = 0$, and $x^{\pm
  p^n}$ for $n \in \Zz$ if $\characteristic(k) = p > 0$.
\begin{definition}
  An \emph{exotic automorphism} of $\kx$ is a definable
  automorphism of $\kx$ which is not $\Lr$-definable in the language of
  rings.
\end{definition}
\begin{lemma} \label{obvious}
  The following are equivalent:
  \begin{enumerate}
  \item Every 1-dimensional definable subgroup of $(\kx)^2$ is
    commensurable with some $Q_{n/m}$.
  \item Every 1-dimensional definable subgroup of $(\kx)^2$ is
    $\Lr$-definable.
  \item There are no exotic automorphisms of $\kx$.
  \end{enumerate}
\end{lemma}
\begin{proof}
  (1)$\implies$(2): suppose $G$ is a 1-dimensional definable subgroup
  of $(\kx)^2$ commensurable with $Q_{n/m}$.  Then $G$ is a union of
  finitely many translates of $Q_{n/m}$, so $G$ is $\Lr$-definable.

  (2)$\implies$(3): If $f : \kx \to \kx$ is a definable automorphism
  of the multiplicative group, then $\Gamma(f) = \{(x,f(x)) : x \in
  \kx\}$ is a 1-dimensional definable subgroup of $(\kx)^2$.  By (2),
  it is $\Lr$-definable, which makes $f$ be $\Lr$-definable.

  (3)$\implies$(1): Assume every definable automorphism of $\kx$ is
  $\Lr$-definable.  Then every definable automorphism has the form
  $f(x) = x^{\pm p^n}$ for some $n \in \Zz$, where $p$ is the
  characteristic exponent of $k$.
  \begin{claim}
    Every definable endomorphism $f : \kx \to \kx$ has the form $f(x)
    = x^r$ for some $r \in \Zz[1/p]$.
  \end{claim}
  \begin{proof}
    If $\ker(f)$ is 1-dimensional, then it is all of $\kx$ (because
    $\kx$ is connected), and so $f$ vanishes, and we can take $r = 0$.
    
    Otherwise, $\ker(f)$ is 0-dimensional, i.e., finite.  Counting
    dimensions, $\dim(\im(f)) = 1$, so $\im(f)$ is all of $\kx$.  That
    is, $f$ is onto.  The only finite subgroups of $\kx$ are the
    cyclic groups $\mu_n = \{x \in \kx : x^n = 1\}$, and so $\ker(f) =
    \mu_n$ for some $n$.  Let $g(x) = x^n$.  Then $g$ and $f$ are
    definable surjective endomorphisms of $\kx$ with the same kernel.
    Therefore $f = h \circ g$ for some definable automorphism $h$.  By
    assumption, $h(x) = x^{\pm p^m}$ for some $m \in \Zz$.  Then
    \begin{equation*}
      f(x) = h(g(x)) = (x^n)^{\pm p^m} = x^{\pm n p^m}.
    \end{equation*}
    Take $r = \pm n p^m$.
  \end{proof}
  Finally, suppose $G$ is a 1-dimensional definable subgroup of
  $(\kx)^2$.  If $G$ has infinite intersection with $1 \times \kx =
  Q_{1/0}$, then $G$ is commensurable with $Q_{1/0}$, by counting
  dimensions.  Suppose $G$ has finite intersection with $1 \times
  \kx$.  This intersection is the kernel of the projection map
  \begin{equation*}
    G \hookrightarrow (\kx)^2 \stackrel{\pi_1}{\to} \kx
  \end{equation*}
  onto the first coordinate.  Since the kernel is finite, the image
  has dimension 1, and must be all of $\kx$.  Therefore, for any $x
  \in \kx$, the slice $G_x = \{y \in \kx : (x,y) \in G\}$ is
  non-empty.

  Let $m = |G \cap (1 \times \kx)|$.  Then $G \cap (1 \times \kx) = 1
  \times \mu_m$.  For any $x \in \kx$, the slice $G_x$ is a coset of
  $\mu_m$, and therefore has the form $\{y \in \kx : y^m = f(x)\}$ for
  some uniquely determined $f(x) \in \kx$.  The map $f$ is a
  homomorphism, so it has the form $f(x) = x^r$ for some rational
  number, by the claim.  Then
  \begin{equation*}
    (x,y) \in G \implies y^m = f(x) = x^r,
  \end{equation*}
  so $G$ is commensurable with $Q_{r/m}$.  In more detail, if $r =
  a/b$ in lowest terms, and $r/m = a/(bm)$ is $a'/b'$ in lowest terms,
  then
  \begin{equation*}
    (x,y) \in G \implies y^m = x^{a/b} \implies y^{bm} = x^a,
  \end{equation*}
  so $G$ is a subgroup of the group $\{(x,y) : y^{bm} = x^a\}$, which
  in turn contains $\{(x,y) : y^{b'} = x^{a'}\} = Q_{a'/b'} =
  Q_{r/m}$.  All these groups are 1-dimensional, so they are all
  commensurable with each other.
\end{proof}
If we are lucky, there are no exotic automorphisms of $\kx$, in which
case the following result applies:
\begin{theorem} \label{rational-slope}
  Let $D \subseteq \RV^2$ be definable.  Let $s$ be a symmetric shape
  such that $\Stab(s)$ is commensurable with $Q_{n/m}$.  Let $D^s =
  \{(a,b) \in \overline{D} : \Shape(D_{a,b}) = s\}$.  Then $D^s$ is a union
  of finitely many points and line segments of slope $n/m$.
\end{theorem}
Here, a \emph{line segment of slope $n/m$} is a set in $\Gamma^2$ of
the form
\begin{equation*}
  \{(x,(n/m)x+c) : x \in (a,b)\}
\end{equation*}
for some open interval $(a,b)$ (possibly half-infinite) and some $c
\in \Gamma$.  A line segment of slope $1/0$ means a set $\{a\} \times
I$ where $a \in \Gamma$ and $I$ is an open interval.  The term ``line
segment'' may be misleading, as we are also including rays and lines.

Theorem~\ref{rational-slope} reduces easily to the following lemma, by
cell decomposition.
\begin{lemma}
  In the setting of Theorem~\ref{rational-slope}, if $D^s$ contains a
  parametric curve $\{(f(z),g(z)) : z \in I\}$ for some open interval
  $I$ and definable continuous functions
  \begin{gather*}
    f : I \to \Gamma \\
    g : I \to \Gamma,
  \end{gather*}
  then the function $mg(z)-nf(z)$ is constant.
\end{lemma}
\begin{proof}
  Suppose $mg(x)-nf(x)$ is non-constant.  Let $G = \Stab(s)$.  By
  assumption, this group is a finite union of cosets of $Q_{n/m}$.
  Fix a central fragment $X \subseteq (\kx)^2$ with $\Shape(X) = s$.
  By Lemma~\ref{symmetric-shape}, $X$ is a finite union of cosets of
  $G$, and therefore a finite union of cosets of $Q_{n/m}$.
  Consequently, the set $\{y^m/x^n : (x,y) \in X\}$ is a finite
  non-empty subset of $\kx$.

  For any $z \in I$, we have $\Shape(D_{f(z),g(z)}) = s$, so the
  fragment $D_{f(z),g(z)}$ is a translate of $X$.  Therefore, the set
  \begin{equation*}
    S_z = \{y^m/x^n : (x,y) \in D_{f(z),g(z)}\}
  \end{equation*}
  is finite and non-empty.  If $(x,y) \in D_{f(z),g(z)}$, then $x \in
  \RV_{f(z)}$ and $y \in \RV_{g(z)}$, so $y^m/x^n \in
  \RV_{mg(z)-nf(z)}$.  Then $\{S_z\}_{z \in I}$ is a bad family.
\end{proof}
By a \emph{line segment}, we mean a vertical interval $\{a\} \times
(b,c)$, or the graph of a linear function $f$:
\begin{equation*}
  \{(x,f(x)) : x \in (a,b)\}.
\end{equation*}
Two line segments $A$ and $B$ \emph{have the same slope} if $A \cap (B
+ (x,y))$ is infinite for some $(x,y) \in \Gamma^2$.  One can show
that this is an equivalence relation on line segments, though we will
not need this.
\begin{theorem} \label{hard-case}
  Let $D \subseteq \RV^2$ be definable.  Let $s$ be a symmetric shape.
  Let $D^s = \{(a,b) \in \overline{D} : \Shape(D_{a,b}) = s\}$.  Then $D^s$
  is a union of finitely many points and finitely many line segments
  of equal slope.
\end{theorem}
The core of the argument is the following lemma:
\begin{lemma} \label{hard-case-lemma}
  In the setting of Theorem~\ref{hard-case}, suppose that $D^s$
  contains the graph of two continuous definable functions
  \begin{gather*}
    f : (a,b) \to \Gamma \\
    g : (a+d,b+d) \to \Gamma.
  \end{gather*}
  Then the function $g(x+d)-f(x)$ is constant on $(a,b)$.
\end{lemma}
\begin{proof}
  Let $G = \Stab(s)$.  If $G$ is commensurable with $Q_{1/0}$, then
  $D^s$ can only contain points and vertical line segments by
  Theorem~\ref{rational-slope}, so $G$ certainly does not contain the
  graph of any function.  Therefore $G$ is not commensurable with
  $Q_{1/0}$.  This implies that the map
  \begin{align*}
    G &\to \kx \\
    (x,y) &\mapsto x
  \end{align*}
  is onto with finite fibers.

  For $x \in (a,b)$, the two points $(x,f(x))$ and $(x+d,g(x+d))$ are
  both in $D^s \subseteq \overline{D}$, and so
  \begin{equation*}
    S_x := \Mov(D_{x,f(x)},D_{x+d,g(x+d)})
  \end{equation*}
  is a coset of $G$ contained in $\RV_d \times \RV_{g(x+d)-f(x)}$.
  Then the map
  \begin{align*}
    S_x &\to \RV_d \\
    (u,v) &\mapsto u
  \end{align*}
  is onto with finite fibers.  Fix some $u \in \RV_d$.  Let $S'_x$ be
  the slice $\{v \in \RV_{g(x+d)-f(x)} : (u,v) \in S_x\}$.  Then
  $S'_x$ is a finite non-empty subset of $\RV_{g(x+d)-f(x)}$.  The
  family $\{S'_x\}_{x \in (a,b)}$ is a bad family, unless
  $g(x+d)-f(x)$ is constant.
\end{proof}
We will also need the following:
\begin{lemma} \label{side-lemma}
  In the setting of Theorem~\ref{hard-case}, if $\Stab(s)$ is not
  commensurable to $Q_{1/0}$, then $D^s$ contains no vertical lines.
\end{lemma}
\begin{proof}
  Let $G = \Stab(s)$.  As in Lemma~\ref{hard-case-lemma}, the
  projection $\pi_1 : G \to \kx$ is onto with finite fibers, because
  $G$ is not commensurable with $Q_{1/0}$.  If $X_0$ is a central
  fragment with shape $s$, then $X_0$ is a finite union of translates
  of $G$ (Lemma~\ref{symmetric-shape}), so the projection $\pi_1 : X_0
  \to \kx$ is again onto with finite fibers.  Finally, if $X \subseteq
  \RV_a \times \RV_b$ is a fragment with shape $s$, then $X$ is a
  translate of $X_0$, and so the projection $\pi_1 : X \to \RV_a$ is
  again onto with finite fibers.

  Suppose for the sake of contradiction that $D^s$ contains a vertical
  interval $\{a\} \times (b,c)$.  Then for any $y \in (b,c)$, we have
  $(a,y) \in D^s$, so $\Shape(D_{a,y}) = s$ and the projection $\pi_1
  : D_{a,y} \to \RV_a$ is onto with finite fibers.  Fix some $u \in
  \RV_a$.  Let $S_y$ be the slice of $D_{a,y}$ over $u$:
  \begin{equation*}
    S_y = \{v \in \RV_y : (u,v) \in D_{a,y}\}.
  \end{equation*}
  Then $S_y$ is a finite nonempty subset of $\RV_y$, and $\{S_y\}_{y
    \in (b,c)}$ is a bad family.
\end{proof}
Now we can prove Theorem~\ref{hard-case}.
\begin{proof}[Proof (of Theorem~\ref{hard-case})]
  Take a cell decomposition $\overline{D} = \coprod_{i = 1}^n C_i$.
  First, we rule out the case where some $C_i$ has dimension 2.  In
  this case, $C_i$ is open.  Take a point $(a,b) \in C_i$ and a small
  $\epsilon > 0$ such that $C_i$ contains the box $(a - \epsilon, a +
  \epsilon) \times (b - \epsilon, b + \epsilon)$.  Then $C_i$ contains
  the graphs of the two functions $f,g : (a-\epsilon,a+\epsilon) \to
  \Gamma$ given by
  \begin{gather*}
    f(x) = b+(x-a) \\
    g(x) = b-(x-a)
  \end{gather*}
  These two functions contradict Lemma~\ref{hard-case-lemma} (with
  $d=0$).

  Therefore, every $C_i$ is one-dimensional.  If $\Stab(s)$ is
  commensurable with $Q_{1/0}$, then $D^s$ is a union of points and
  vertical line segments by Theorem~\ref{rational-slope}, and there is
  nothing to prove.  So we may assume that $\Stab(s)$ is not
  commensurable with $Q_{1/0}$.  Then Lemma~\ref{side-lemma} shows
  that $D^s$ contains no vertical lines segments.  Therefore every
  cell is either a point, or the graph $\Gamma(f)$ of a continuous
  definable function $f : (a,b) \to \Gamma$.
  \begin{claim}
    If $\Gamma(f) \subseteq D^s$, then $f$ is linear.
  \end{claim}
  \begin{proof}
    For any small $\epsilon > 0$, let $g$ and $h$ be the restrictions
    of $f$ to the intervals $(a,b-\epsilon)$ and $(a+\epsilon,b)$,
    respectively.  The graphs of $g$ and $h$ are contained in $D^s$.
    By Lemma~\ref{hard-case-lemma}, the following function must be
    constant:
    \begin{gather*}
      (a,b-\epsilon) \to \Gamma \\
      x \mapsto h(x+\epsilon) - g(x) = f(x+\epsilon) - f(x) =
      \delta_\epsilon f (x).
    \end{gather*}
    Then $f$ is linear, by Corollary~\ref{reduce-seven}.
  \end{proof}
  Now $D^s$ is a finite union of points and graphs of linear
  functions.  If
  \begin{gather*}
    f : (a,b) \to \Gamma \\
    g : (c,d) \to \Gamma
  \end{gather*}
  are two linear functions whose graphs are contained in $D^s$, then
  the graphs must have the same slope by Lemma~\ref{hard-case-lemma}.
  In more detail, we can shrink $b$ or $d$ to make $b-a = d-c$.  Let
  $\epsilon = c-a = d-b$, so that $\dom(g) = (a+\epsilon,b+\epsilon)$.
  Then Lemma~\ref{hard-case-lemma} implies that
  \begin{equation*}
    g(x+\epsilon)-f(x) \text{ is a constant } u,
  \end{equation*}
  in which case $\Gamma(g) = (\epsilon,u) + \Gamma(f)$.  Thus, the two
  line segments have the same slope.
\end{proof}
Putting everything together, we have proven the following:
\begin{theorem} \label{key-thm}
  Let $D$ be a definable subset of $\RV^2$ which is small over
  $\Gamma^2$, in the sense that the fibers of $D \hookrightarrow \RV^2
  \to \Gamma^2$ have dimension at most 1.
  \begin{enumerate}
  \item $\overline{D}$ is a union of finitely many line segments.
  \item If $\kx$ has no exotic automorphisms, then $\overline{D}$ is a
    union of finitely many line segments with rational slopes.
  \item For all but finitely many points of $\overline{D}$, the
    fragment $D_{x,y}$ is a finite union of translates of a
    1-dimensional subgroup of $(\kx)^2$.
  \end{enumerate}
\end{theorem}
Indeed, if $\{s_1,\ldots,s_n\} = \{\Shape(D_{x,y}) : (x,y) \in
\overline{D}\}$, then $\overline{D}$ is the union of the finitely many
sets $D^{s_1} \cup \cdots \cup D^{s_n}$.  Each of the sets $D^{s_i}$
is a finite union of points and line segments by
Theorems~\ref{asymcase} and \ref{hard-case}.  In the absence of exotic
automorphisms, we can assume that the slopes are rational
(Theorem~\ref{rational-slope} plus Lemma~\ref{obvious}).  When $s_i$
is asymmetric, $D^{s_i}$ is finite (Theorem~\ref{asymcase}), and so
$\Shape(D_{x,y})$ is symmetric at all but finitely many $(x,y)$ in
$\overline{D}$.  Then $D_{x,y}$ must be a finite union of cosets by
Lemma~\ref{symmetric-shape}.

We also note the following connection to dp-rank:
\begin{theorem} \label{dpr1}
  Let $D$ be a definable subset of $\RV^2$.  Then $\dpr(D) \le 1$ if
  and only if $D$ is small over $\Gamma^2$.
\end{theorem}
\begin{proof}
  First suppose $\dpr(D) \le 1$.  Then the subset $D_{a,b} = D \cap
  (\RV_a \times \RV_b)$ has dp-rank at most 1 for any $(a,b) \in
  \Gamma^2$.  Dp-rank agrees with dimension and Morley rank on $\RV_a
  \times \RV_b$, so $\dim(D_{a,b}) \le 1$, and $D$ is small over
  $\Gamma^2$.

  Conversely, suppose $D$ is small over $\Gamma^2$.  By
  Lemma~\ref{finshapes}, we can partition $\overline{D}$ into finitely
  many sets $D^{s_1} \sqcup \cdots \sqcup D^{s_n}$ such that
  $\Shape(D_{x,y}) = s_i$ for $(x,y) \in D^{s_i}$.  Let $D_i$ be the
  part of $D$ lying over $D^{s_i}$.  It suffices to show that
  $\dpr(D_i) \le 1$ for each $i$.  Replacing $D$ with $D_i$, we may
  assume that $\overline{D} = D^s$ for some fixed shape $s$, i.e.,
  $\Shape(D_{x,y}) = s$ for every $(x,y) \in \overline{D}$.

  If $s$ is asymmetric, then $\overline{D} = D^s$ is finite by
  Theorem~\ref{asymcase}.  Therefore, $D$ is a finite union
  $\bigcup_{i = 1}^n D_{a_i,b_i}$, where
  $\{(a_1,b_1),\ldots,(a_n,b_n)\}$ enumerates $\overline{D}$.  Each of
  the fragments $D_{a_i,b_i}$ has dimension $\le 1$ because $D$ is
  small over $\Gamma^2$, and dimension agrees with dp-rank on
  $\RV_{a_i} \times \RV_{b_i}$, so $\dpr(D_{a_i,b_i}) \le 1$ and
  $\dpr(D) \le 1$.

  Next suppose that $s$ is symmetric.  Exchanging the two coordinates,
  we may assume that $s$ is not commensurable with $Q_{1/0}$.  Then
  $\overline{D} = D^s$ contains no vertical lines by
  Lemma~\ref{side-lemma}, and it is one-dimensional by
  Theorem~\ref{hard-case}.  Therefore, the projection $\pi_1 :
  \overline{D} \to \Gamma$ has finite fibers.  Moreover, for any
  $(a,b) \in \overline{D}$, the projection $\pi_1 : D_{a,b} \to \RV_a$
  has finite fibers as in the proof of Theorem~\ref{hard-case}.  It
  follows that the projection $\pi_1 : D \to \RV$ has finite fibers.
  Otherwise, $D$ contains an indiscernible sequence of distinct points
  $(x,y_1), (x,y_2), \ldots$.  Let $a = v(x)$ and $b_i = v(y_i)$, so
  that $(x,y_i) \in \RV_a \times \RV_{b_i}$.  Then $(a,b_i) \in
  \overline{D}$.  Because $\pi_1 : \overline{D} \to \Gamma$ has finite
  fibers, the set $\{b_1,b_2,\ldots\}$ is finite, and so $b_i$ is a
  constant $b$ by indiscernibility.  Then $(x,y_i)$ is an infinite
  sequence of distinct points in $D_{a,b}$, contradicting the fact
  that $\pi_1 : D_{a,b} \to \RV_a$ has finite fibers.

  So $\pi_1 : D \to \RV$ has finite fibers, as claimed.  By properties
  of dp-rank, $\dpr(D) \le \dpr(\RV) \le \dpr(\Mm) = 1$.
\end{proof}

\section{Asymptotics and polynomial-boundedness}
\begin{theorem} \label{affine-to-linear}
  Let $f : (a,+\infty) \to \Gamma$ be a definable linear function.
  Then $f(x) = g(x)+c$ for some definable endomorphism $g : \Gamma
  \to \Gamma$ and constant $c \in \Gamma$.
\end{theorem}
\begin{proof}
  For $b \in \Gamma$, define $g(b)$ to be $f(x+b)-f(x)$ for any $x$
  with $\{x,x+b\} \subseteq (a,+\infty)$.  The choice of $x$ doesn't
  matter because $f$ is linear.  The definable function $g$ is an
  endomorphism because given $b_1, b_2$, we have
  \begin{equation*}
    g(b_1+b_2) = f(x+b_1+b_2)-f(x) =
    [f(x+b_1+b_2)-f(x+b_2)]+[f(x+b_2)-f(x)] = g(b_1)+g(b_2)
  \end{equation*}
  for all sufficiently large $x$.  Note that for $x,y >
  a$, we have
  \begin{gather*}
    f(x)-f(y) = g(x-y) = g(x)-g(y) \\
    f(x)-g(x) = f(y)-g(y).
  \end{gather*}
  Therefore $f(x)-g(x)$ is some constant $c$.
\end{proof}
\begin{lemma} \label{asymptotics}
  Let $f : \Mm^\times \to \RV$ be definable.
  \begin{enumerate}
  \item There is a ball $B \ni 0$, an endomorphism $g : \Gamma \to
    \Gamma$ and a constant $c \in \Gamma$ such that
    \begin{equation*}
      v(f(x)) = g(v(x)) + c \text{ for } x \in B \setminus \{0\}.
    \end{equation*}
  \item If $k^\times$ has no exotic automorphisms, then $g(x) = qx$
    for some $q \in \Qq$.
  \end{enumerate}
\end{lemma}
\begin{proof}
  Let $D = \{(\rv(x),f(x)) : x \in \Mm^\times\}$.  This set has
  dp-rank 1 as it is the image of $\Mm^\times$ under a definable map.
  By Theorem~\ref{dpr1}, $D$ is small over $\Gamma^2$.  The image in
  $\Gamma^2$ is
  \begin{equation*}
    \overline{D} = \{(v(x),v(f(x))) : x \in \Mm^\times\}.
  \end{equation*}
  By Theorem~\ref{key-thm}, $\overline{D}$ has a cell decomposition
  $\coprod_{i = 1}^n C_i$ where each $C_i$ is a point or a ``line
  segment''---a vertical line segment or the graph of a linear
  function.  Moreover, if $k^\times$ has no exotic automorphisms then
  we can assume each line segment has rational slope.

  As in the proof of Theorem~\ref{copy-1}, there is a cell $C_i$ and a
  ball $B \ni 0$ such that
  \begin{equation*}
    x \in B \setminus \{0\} \implies (v(x),v(f(x))) \in C_i,
  \end{equation*}
  and $C_i$ must be the graph of a continuous definable function $h :
  (a,+\infty) \to \Gamma$, so that
  \begin{equation*}
    x \in B \setminus \{0\} \implies v(f(x)) = h(v(x)).
  \end{equation*}
  Because $C_i$ is a line segment, $h$ must be a linear map, so $h(x)
  = g(x)+c$ for some endomorphism $g : \Gamma \to \Gamma$, by Lemma~\ref{affine-to-linear}.
  If $k^\times$ has no exotic automorphisms, then $C_i$ has rational
  slope, so $g(x) = q \cdot x$ for some $q \in \Qq$.
\end{proof}
Say that ``$\Gamma$ has no exotic automorphisms'' if the only
definable automorphisms of the ordered abelian group $(\Gamma,+,\le)$
are multiplication by positive rational numbers.  This is equivalent
to the statement that the only definable endomorphisms $\Gamma \to
\Gamma$ are $g(x) = qx$ for rational $q$.
\begin{theorem} \label{pbound}
  If $k^\times$ has no exotic automorphisms or $\Gamma$ has no exotic
  automorphisms, then $\Mm$ is polynomially bounded: for any definable
  function $f : \Mm \to \Mm$, there is some $n$ such that
  \begin{equation*}
    v(f(x)) > n \cdot v(x)
  \end{equation*}
  for all $x$ outside a bounded set.
\end{theorem}
\begin{proof}
  Applying Lemma~\ref{asymptotics} to the function $\rv(f(1/x))$, we
  get a punctured ball $B \setminus \{0\}$ and a rational number $q$
  such that
  \begin{equation*}
    x \in B \setminus \{0\} \implies v(f(1/x)) = q \cdot v(x) = -q
    \cdot v(1/x).
  \end{equation*}
  This implies that $f$ is polynomially bounded.
\end{proof}

\section{The exchange property}
The group $\Mm/\Oo$ is the set of closed balls of valuative radius 0.
The zero element $0 \in \Mm/\Oo$ ``is'' the valuation ring $\Oo$.  If $B \in
\Mm/\Oo \setminus \{0\}$, define $\rv(B)$ to be $\rv(x)$ for any element
$x \in B$.  This is well-defined, because if $x,y\in B$, then
\begin{gather*}
  x-y \in \Oo \text{ but } x,y \notin \Oo \\
  v(x-y) \ge 0 > v(y) \\
  \rv(x) = \rv(y).
\end{gather*}
Similarly, define $v(B)$ to be $v(x)$ for any $x \in B$.  This makes
the following diagram commute:
\begin{equation*}
  \xymatrix{ \Mm/\Oo \setminus \{0\} \ar[dr]_-v \ar[r]^-{\rv} & \RV \ar[d]
    \\ & \Gamma,}
\end{equation*}
where $\RV \to \Gamma$ is the usual map.  Note that $v(B) < 0$ for
non-zero $B \in \Mm/\Oo$.

We can apply the ideas from the proofs of Theorems~\ref{copy-2} and
\ref{pbound} to the setting of definable functions $\Mm \to \Mm/\Oo$,
yielding the following:
\begin{lemma}\label{shuksan}
  Let $D \subseteq \Mm$ be definable and infinite and $f : D \to
  \Mm/\Oo$ be definable and injective.
  \begin{enumerate}
  \item \label{sk1} There is a ball $B_1 \subseteq D$ and a definable
    continuous function $h : \Gamma \to \Gamma$ such that
    \begin{equation*}
      v(f(x)-f(y)) = h(v(x-y))
    \end{equation*}
    for distinct $x,y \in B_1$.
  \item There is a ball $B_2 \subseteq D$, a definable endomorphism $g :
    \Gamma \to \Gamma$, and a constant $c \in \Gamma$ such that
    \begin{equation*}
      v(f(x)-f(y)) = g(v(x-y)) + c
    \end{equation*}
    for distinct $x,y \in B_2$.
  \end{enumerate}
\end{lemma}
Later in Lemma~\ref{glacier}, we will see that the assumptions of
Lemma~\ref{shuksan} can never hold.
\begin{proof}
  \begin{enumerate}
  \item The proofs of Theorems~\ref{copy-1} and \ref{copy-2}
    work in this setting.
  \item Let $h$ and $B_1$ be as in part (\ref{sk1}).  Fix some $a \in
    B_1$.  Applying Lemma~\ref{asymptotics} to the map
    \begin{equation*}
      y \mapsto \rv(f(y+a)-f(a)),
    \end{equation*}
    we get a definable endomorphism $g : \Gamma \to \Gamma$, a constant
    $c \in \Gamma$, and a ball $B_2 \ni a$ such that
    \begin{gather*}
      y \in B_2 - a \implies v(f(y+a)-f(a)) = g(v(y)) + c \\
      x \in B_2 \implies v(f(x)-f(a)) = g(v(x-a)) + c.
    \end{gather*}
    On the other hand,
    \begin{gather*}
      x \in B_2 \implies v(f(x)-f(a)) = h(v(x-a)).
    \end{gather*}
    Then $h(z) = g(z)+c$ for any $z$ greater than the radius of $B_2$.  It
    follows that
    \begin{equation*}
      v(f(x)-f(y)) = h(v(x-y)) = g(v(x-y)) + c
    \end{equation*}
    for \emph{any} $x,y \in B_2$. \qedhere
  \end{enumerate}
\end{proof}
\begin{lemma} \label{baker}
  Suppose $D \subseteq \Mm/\Oo$ is a definable set and $c \in \Gamma$
  is a constant such that
  \begin{equation*}
    v(x-y) = c \text{ for distinct } x,y \in D.
  \end{equation*}
  Then $D$ is finite.
\end{lemma}
\begin{proof}
  Suppose $D$ is infinite.  Note that $c = v(x-y) < 0$.  Let $U$ be
  the open ball of radius $c$.  Then $U/\Oo$ is infinite.  The map
  \begin{align*}
    D \times (U/\Oo) \to \Mm/\Oo \\
    (x,z) \mapsto x+z
  \end{align*}
  is injective.\footnote{Indeed, if $(x,z)$ and $(y,w)$ map to the
  same point, then $x+z = y+w$, $x-y = w-z \in U/\Oo$, $v(x-y) > c$,
  $x=y$, and $z=w$.}  But $D \times (U/\Oo)$ has dp-rank at least 2,
  and $\Mm/\Oo$ has dp-rank at most 1.
\end{proof}
\begin{lemma} \label{glacier}
  If $D \subseteq \Mm$ is infinite and definable, then there is
  \emph{no} definable injection $f : D \to \Mm/\Oo$.
\end{lemma}
\begin{proof}
  By Lemma~\ref{shuksan}, there is a ball $B \subseteq D$, an endomorphism $g : \Gamma \to \Gamma$, and a constant $c$ such that
  \begin{equation*}
    v(f(x)-f(y)) = g(v(x-y)) + c
  \end{equation*}
  for distinct $x,y \in B$.  By o-minimality of $\Gamma$, the endomorphism $g$ must fall into one of three cases:
  \begin{enumerate}
  \item $g$ is increasing and $\lim_{z \to +\infty} g(z) = +\infty$.
    Fix $a \in B$.  As $x$ approaches $a$, the valuation $v(x-a)$
    approaches $+\infty$, and so
    \begin{equation*}
      \lim_{x \to a} v(f(x)-f(a)) = \lim_{x \to a} \left(g(v(x-a)) +
      c\right) = +\infty.
    \end{equation*}
    But $f(x)-f(a)$ is a non-zero element of $\Mm/\Oo$, so its
    valuation is negative, and the limit cannot approach $+\infty$.
  \item $g$ is decreasing and $\lim_{z \to +\infty} g(z) = -\infty$.
    Take $x,y,z \in B$ with
    \begin{equation*}
      v(x-y) > v(x-z) = v(y-z).
    \end{equation*}
    Applying the order-reversing function $g(z)+c$, we get
    \begin{equation*}
      v(f(x)-f(y)) < v(f(x)-f(z)) = v(f(y)-f(z)).
    \end{equation*}
    This is impossible.  For example, taking points $x',y',z'$ in
    $f(x),f(y),f(z)$, respectively, we get
    \begin{equation*}
      v(x'-y') < v(x'-z') = v(y'-z'),
    \end{equation*}
    a violation of the triangle inequality.
  \item $g$ is identically zero.  Then
    \begin{equation*}
      v(f(x)-f(y)) = c \text{ for distinct } x,y \in B. \tag{$\ast$}
    \end{equation*}
    Let $D = \{f(x) : x \in B\}$.  Then the set $D$ contradicts
    Lemma~\ref{baker}. \qedhere
  \end{enumerate}
\end{proof}
\begin{theorem} \label{big-thm}
  $\Mm$ has the exchange property.
\end{theorem}
\begin{proof}
  Suppose not.  By \cite[Proposition~6.1]{cminfields}, there is a
  ``bad function''.  Whatever the definition of bad function is, it at
  least gives us a definable bijection between an infinite definable
  set $D \subseteq \Mm$, and an infinite definable antichain in the
  class $\mathcal{B}$ of closed balls.  That is, there is a definable
  injection $f : D \to \mathcal{B}$ such that
  \begin{equation*}
    f(x) \cap f(y) = \varnothing \text{ for distinct } x,y \in B.
  \end{equation*}
  Let $g(x)$ be the radius of $f(x)$.  By Fact~\ref{const-to-gamma}, the
  function $g$ is locally constant almost everywhere.  Shrinking $D$,
  we may assume that $g$ is constant.  Then $f$ is a definable
  injection from $D$ to the class of closed balls of radius $\gamma$,
  for some fixed $\gamma \in \Gamma$.  Rescaling, we may assume
  $\gamma = 0$.  Then $f$ is a definable injection from $B$ to
  $\Mm/\Oo$, contradicting Lemma~\ref{glacier}.
\end{proof}
\begin{remark}
  The definition of ``bad function'' gives much more than an injection
  from $D$ to an antichain in $\mathcal{B}$.  One even gets an
  isomorphism of C-structures.  Using this, one can probably
  streamline the proof of Theorem~\ref{big-thm}.  For example, one can
  essentially avoid Lemma~\ref{baker} as well as Cases 2 and 3 in the
  proof of Lemma~\ref{glacier}.  We have avoided this streamlined
  approach in order to minimize the use of C-structures.
\end{remark}
\begin{corollary}
  The theory $T$ is geometric.
\end{corollary}
\begin{proof}
  We need to check that $T$ has the exchange property and eliminates
  $\exists^\infty$.  The first part is Theorem~\ref{big-thm}.  The
  second part is a trivial consequence of $C$-minimality.  Indeed, $D
  \subseteq \Mm$ is infinite if and only if it contains a ball.
\end{proof}
\subsection{A variant of Lemma~\ref{shuksan}}\label{other-assaf}
Assaf Hasson suggested the following variant of Lemma~\ref{shuksan}:
\begin{proposition} \label{for-hasson}
  Let $D \subseteq \Mm/\Oo$ be definable and infinite, and let $f : D
  \to \Mm/\Oo$ be definable and injective.  Then there is a ball $B
  \subseteq D$ and a linear function $g : [-\gamma,\gamma] \to \Gamma$
  such that $g(0) = 0$, and
  \begin{equation*}
    v(f(x) - f(y)) = g(v(x-y)) \tag{$\dag$}
  \end{equation*}
  for distinct $x,y \in B$.  If $k^\times$ has no exotic
  automorphisms, then we can take $g$ to be $g(x) = q \cdot x$ for
  some rational $q$.
\end{proposition}
\begin{proof}
  Because $\Mm/\Oo$ is itself $C$-minimal, the analogue of
  Fact~\ref{const-to-gamma} holds: any definable function $D \to
  \Gamma^n$ is locally constant at all but finitely many points of
  $D$.  Then Corollary~\ref{corfam} and
  Theorem~\ref{copy-1}--\ref{copy-2} continue to hold in this setting.
  In particular, we can shrink $D$ and arrange for ($\dag$) to hold
  for some definable function $g : (-\gamma,0) \to \Gamma_{<0}$, not
  necessarily linear.  By the generic continuity of definable
  functions in $C$-minimal structures (applied to the $C$-minimal
  structure $\Mm/\Oo$), we can shrink $D$ and assume that $f$ is
  continuous.  Since the ``value set'' of $\Mm/\Oo$ is $\Gamma_{<0}$,
  continuity of $f$ means precisely that
  \begin{equation*}
    \lim_{\gamma \to 0^-} g(\gamma) = 0.
  \end{equation*}
  Fix some $a \in D$.  The set $\{(\rv(x-a),\rv(f(x)-f(a))) : x \in
  D\} \subseteq \RV \times \RV$ has dp-rank 1, so its image in
  $\Gamma \times \Gamma$, the set
  \begin{equation*}
    S = \{(v(x-a),v(f(x)-f(a))) : x \in D\} \subseteq \Gamma \times \Gamma,
  \end{equation*}
  must be a finite union of points and line segments, and the line
  segments must have rational slopes unless $\kx$ has exotic
  automorphisms.  The set $S$ is contained in the graph of $g$.  As in
  the proof of Lemma~\ref{asymptotics}, it follows that $g$ is
  piecewise linear.  In particular, $g$ is linear on an interval of
  the form $(-\gamma,0)$, for sufficiently small $\gamma > 0$.  An
  easy exercise shows that $g$ extends to a linear function on the
  interval $[-\gamma,\gamma]$.  The fact that $\lim_{\gamma \to 0^-}
  g(\gamma) = 0$ implies that $g(0) = 0$.  In the case where $\kx$ has
  no exotic automorphisms, $g$ must be $\Qq$-linear, so $g(x) = q
  \cdot x$ for some rational number $q$.
\end{proof}
%

\section{Definably complete $C$-minimal fields}
In this section, \textbf{suppose that $\Mm$ is definably complete},
meaning that if $\{B_a\}_{a \in D}$ is a definable chain of balls with
radii approaching $+\infty$, then $\bigcap_{a \in D} B_a$ is
non-empty.
\begin{lemma} \label{hozomeen}
  Let $f : \Mm^\times \to \Mm$ be definable.  Let $B$ be a closed
  ball.  Suppose that $f(x) \in B$ for all sufficiently small $x$.
  Then there is an open ball $B' \subset B$ such that $f(x) \in B'$
  for all sufficiently small $x$.
\end{lemma}
\begin{proof}
  Without loss of generality, $B = \Oo$.  For $\gamma \in \Gamma$, let
  \begin{equation*}
    A_\gamma = \{\res(f(x)) : x \in \Mm, ~ v(x) > \gamma\}.
  \end{equation*}
  Then $A_\gamma$ is a descending chain of non-empty subsets of $k$.
  By the orthogonality of $\Gamma$ and $k$
  (Proposition~\ref{prop-ortho}), there are only finitely many
  $A_\gamma$.  So there is some fixed definable set $A \subseteq k$
  such that $A_\gamma = A$ for all sufficiently large $\gamma$.  Take
  $\alpha \in A$.  Then for every $\gamma$, there is $x$ with $v(x) >
  \gamma$ and $\res(f(x)) = \alpha$.  That is, the set
  \begin{equation*}
    \{x : \res(f(x)) = \alpha\}
  \end{equation*}
  contains arbitrarily small points.  By Fact~\ref{inf-fact},
  this set contains \emph{all} sufficiently small $x$.  Then we can
  take $B'$ to be the residue class $\{y \in \Oo : \res(y) =
  \alpha\}$.
\end{proof}
\begin{lemma} \label{fernow}
  Let $\mathcal{C}$ be a non-empty definable chain of balls with empty
  intersection.  There is a unique definable function $\rho : \Mm \to
  \RV$ with the following property:
  \begin{enumerate}
  \item If $a \in \Mm$, $B \in \mathcal{C}$, $a \notin B$, and $x \in
    B$, then $\rv(a-x) = \rho(a)$.
  \end{enumerate}
  Moreover, $v(\rho(a))$ has the following two properties:
  \begin{enumerate}[resume]
  \item \label{fern2} If $v(\rho(a)) \ge r$, then there is a ball $B
    \in \mathcal{C}$ with radius $\ge r$.
  \item \label{fern3} If $B \in \mathcal{C}$ has radius $r$, then $v(\rho(a)) \ge r$
    for $a \in B$.
  \end{enumerate}
\end{lemma}
Intuitively, $\rho(a)$ is $\rv(a-x)$ for a non-existent point $x$ in
the intersection $\bigcap \mathcal{C}$, and $v(\rho(a))$ measures how
close $a$ is to the intersection.
\begin{proof}
  \begin{enumerate}
  \item Fix a point $a \in \Mm$.  Then $a \notin \varnothing = \bigcap
    \mathcal{C}$, so there is at least one ball $B \in \mathcal{C}$
    with $a \notin B$.  Take any $x \in B$ and define
    \begin{equation*}
      \rho(a,B) = \rv(a-x).
    \end{equation*}
    This is well-defined because if $x,y \in B$, then
    \begin{equation*}
      v(x-y) > v(a-x) = v(a-y) \tag{$\ast$}
    \end{equation*}
    thanks to the ball $B$ which contains $x$ and $y$ but not $a$.
    But ($\ast$) implies $\rv(a-x) = \rv(a-y)$.

    Moreover, if $B$ and $B'$ are two different balls in $\mathcal{C}$
    not containing $a$, then taking $x \in B \cap B'$ we see
    \begin{equation*}
      \rho(a,B) = \rv(a-x) = \rho(a,B').
    \end{equation*}
    Therefore $\rho(a,B)$ does not depend on $B$, and we can write it
    as $\rho(a)$.  Definability of $\rho$ is straightforward.
  \item Suppose $v(\rho(a)) \ge r$.  Take a ball $B \in \mathcal{C}$ not containing
    $a$.  If the radius of $B$ is at least $r$, we are done.
    Otherwise, the radius is less than $r$, so there are $x,y \in B$
    with $v(x-y) < r$.  Then
    \begin{equation*}
      \rv(a-x) = \rho(a) = \rv(a-y)
    \end{equation*}
    implying that
    \begin{equation*}
      v(a-x) < v(x-y) < r.
    \end{equation*}
    But $\rho(a) = \rv(a-x)$, so
    \begin{equation*}
      v(\rho(a)) = v(a-x) < r \le v(\rho(a)),
    \end{equation*}
    a contradiction.
  \item Suppose $B \in \mathcal{C}$ has radius $r$ and $a \in B$.
    Take $B' \in \mathcal{C}$ such that $a \notin B'$.  Then $B'
    \subsetneq B$.  Take $x \in B'$.  By definition of $\rho$,
    \begin{gather*}
      \rho(a) = \rv(a-x) \\
      v(\rho(a)) = v(a-x) \ge r. \qedhere
    \end{gather*}
  \end{enumerate}
\end{proof}
\begin{theorem} \label{limits-exist}
  Let $B \ni 0$ be a ball and $f : B \setminus \{0\} \to \Mm$ be a
  definable function.  Then $\lim_{x \to 0} f(x)$ exists in $\Mm \cup
  \{\infty\}$.
\end{theorem}
Here, $\lim_{x \to 0} f(x) = \infty$ means that $\lim_{x \to 0} 1/f(x)
= 0$.  Theorem~\ref{limits-exist} was proved by Cubides Kovacsics and
Delon \cite[Lemma~5.3]{CK-D} under an additional assumption on
$\Gamma$, building off an earlier result of Delon
\cite[Proposition~4.5]{honneur}.\footnote{It's possible that Cubides
Kovacsics and Delon's proof would apply here with minimal changes.
Unfortunately, I wasn't able to understand the proof of
\cite[Proposition~4.5]{honneur}.}
\begin{proof}
  By Corollary~\ref{inf-cor}, we may shrink $B$ and assume that one of
  the following holds:
  \begin{itemize}
  \item $f(x) \in \Oo$ for all $x$.
  \item $f(x) \notin \Oo$ for all $x$.
  \end{itemize}
  Replacing $f(x)$ with $1/f(x)$, we may assume the first case: $f(x)
  \in \Oo$.

  Let $\mathcal{C}$ be the definable class of (open and closed) balls
  $B'$ with the property that $f(x) \in B'$ for all sufficiently small
  $x$.  We have just arranged $\Oo \in \mathcal{C}$, and so
  $\mathcal{C}$ is non-empty.  Moreover, $\mathcal{C}$ is a chain,
  because if $B_1, B_2$ are two incomparable balls in $\mathcal{C}$,
  then $f(x) \in B_1 \cap B_2 = \varnothing$ for sufficiently small
  $x$.

  First suppose that $\bigcap \mathcal{C}$ is non-empty, containing a
  point $a$.  By Lemma~\ref{asymptotics} applied to $\rv(f(x)-a)$,
  there is an endomorphism $g : \Gamma \to \Gamma$ and constant $c \in
  \Gamma$ such that
  \begin{equation*}
    v(f(x)-a) = g(v(x)) + c
  \end{equation*}
  for all sufficiently small $x$.  There are three cases:
  \begin{enumerate}
  \item $g$ is increasing and $\lim_{z \to +\infty} g(z) = +\infty$.  Then
    \begin{equation*}
      \lim_{x \to 0} v(f(x)-a) = \lim_{x \to 0} (g(v(x)) + c) = +\infty,
    \end{equation*}
    and so $\lim_{x \to 0} f(x) = a$.
  \item $g$ is decreasing and $\lim_{z \to +\infty} g(z) = -\infty$.
    Then $\lim_{x \to 0} v(f(x) - a) = -\infty$, contradicting the
    fact that $f(x) \in \Oo$.
  \item $g$ is zero and so
    \begin{equation*}
      v(f(x)-a) = c \text{ for all sufficiently small $x$.} \tag{$\ast$}
    \end{equation*}
    Let $B' = \{y \in \Mm : v(y-a) \ge c\}$.  Then $f(x) \in B'$ for
    all sufficiently small $x$.  By Lemma~\ref{hozomeen}, there is a
    smaller open ball $B'' \subset B'$ such that $f(x) \in B''$ for
    sufficiently small $x$.  Then $B'' \in \mathcal{C}$, so $a \in
    B''$.  This forces $B''$ to be the open ball $\{y \in \Mm : v(y-a)
    > c\}$, or a subset.  Thus
    \begin{equation*}
      v(f(x)-a) > c \text{ for all sufficiently small $x$},
    \end{equation*}
    contradicting ($\ast$).
  \end{enumerate}
  Otherwise, $\bigcap \mathcal{C}$ is empty.  Consequently,
  $\mathcal{C}$ contains no minimum.  Let $r_\infty$ be the supremum
  of the radii of balls in $\mathcal{C}$.  By definable completeness,
  $r_\infty < +\infty$.  In other words, $\mathcal{C}$ witnesses a
  failure of definable spherical completeness, but not a failure of
  definable completeness.  No ball in $\mathcal{C}$ has radius exactly
  $r_\infty$, or else $\mathcal{C}$ would contain a minimum.

  Let $\rho : \Mm \to \RV$ be as in Lemma~\ref{fernow}.  By
  Lemma~\ref{fernow}(\ref{fern2}),
  \begin{equation*}
    v(\rho(a)) < r_\infty \text{ for } a \in \Mm.  \tag{$\dag$}
  \end{equation*}
  On the other hand,
  \begin{claim}
    If $r < r_\infty$, then $v(\rho(f(x))) \ge r$ for all sufficiently
    small $x$.
  \end{claim}
  \begin{claimproof}
    By definition of $r_\infty$, there is a ball $B \in \mathcal{C}$
    with radius at least $r$.  By definition of $\mathcal{C}$, we have
    $f(x) \in B$ for all sufficiently small $x$.  Then $v(\rho(f(x)))
    \ge r$ by Lemma~\ref{fernow}(\ref{fern3}).
  \end{claimproof}
  Therefore, as $x$ approaches 0, $v(\rho(f(x)))$ approaches
  $r_\infty$ from below.

  By Lemma~\ref{asymptotics} applied to $\rho(f(x))$, there is a definable endomorphism $g : \Gamma \to \Gamma$ and constant $c
  \in \Gamma$ such that
  \begin{equation*}
    v(\rho(f(x))) = g(v(x)) + c.
  \end{equation*}
  As $x$ goes to zero, the left hand side approaches $r_\infty$.  This
  can only happen if $g$ vanishes and $c = r_\infty$.  Then
  \begin{equation*}
    v(\rho(f(x))) = r_\infty
  \end{equation*}
  for all sufficiently small $x$, \emph{contradicting ($\dag$)}.
\end{proof}
Combining this with the results on asymptotics, we get an interesting
corollary:
\begin{theorem} \label{twist}
  Let $f : \Mm^\times \to \Mm^\times$ be definable.  Then there is a
  definable multiplicative endomorphism $g : \Mm^\times \to \Mm^\times$ and a
  constant $c \in \Mm^\times$ such that
  \begin{equation*}
    \lim_{x \to 0} f(x)/g(x) = c
  \end{equation*}
\end{theorem}
\begin{proof}
  Applying Lemma~\ref{asymptotics} to $\rv(f(x))$, we get an
  endomorphism $h : \Gamma \to \Gamma$ and constant $\gamma$ such that
  \begin{equation*}
    v(f(x)) = h(v(x)) + \gamma
  \end{equation*}
  for sufficiently small $x \ne 0$.  For any constant $a \in
  \Mm^\times$, we then have
  \begin{gather*}
    v\left(\frac{f(ax)}{f(x)}\right) = v(f(ax)) - v(f(x)) = h(v(ax)) -
    h(v(x)) = h(v(ax)-v(x)) = h(v(a))
  \end{gather*}
  for sufficiently small $x$.  It follows that $f(ax)/f(x)$ is
  confined to the annulus $\{y \in \Mm^\times : v(y) = h(v(a))\}$.  Define
  \begin{equation*}
    g(a) := \lim_{x \to 0} \frac{f(ax)}{f(x)}.
  \end{equation*}
  By Theorem~\ref{limits-exist}, $g(a)$ exists and lives in the
  annulus.  In particular, $g(a) \in \Mm^\times$ rather than $g(a) =
  \infty$ or $g(a) = 0$.  Moreover,
  \begin{equation*}
    v(g(a)) = h(v(a)).  \tag{$\ast$}
  \end{equation*}
  Given $a,b \in \Mm^\times$,
  \begin{equation*}
    g(a) = \lim_{x \to 0} \frac{f(abx)}{f(bx)}
  \end{equation*}
  and so
  \begin{equation*}
    g(ab) = \lim_{x \to 0} \left(\frac{f(abx)}{f(bx)} \cdot
    \frac{f(bx)}{f(x)}\right) = g(a)g(b).
  \end{equation*}
  Thus $g$ is a definable multiplicative endomorphism.  Finally, note that
  \begin{equation*}
    v\left(\frac{f(x)}{g(x)}\right) = v(f(x))-v(g(x)) = h(v(x)) +
    \gamma - h(v(x)) = \gamma
  \end{equation*}
  for sufficiently small $x$, by ($\ast$).  Applying
  Theorem~\ref{limits-exist} again,
  \begin{equation*}
    \lim_{x \to 0} \frac{f(x)}{g(x)} = c
  \end{equation*}
  for some $c \in \Mm^\times$ with $v(c) = \gamma$.
\end{proof}
It is worth noting the following about definable endomorphisms of
$\Mm^\times$:
\begin{theorem} \label{chikamin}
  Let $f : \Mm^\times \to \Mm^\times$ be a definable endomorphism.
  \begin{enumerate}
  \item There is a definable endomorphism $g : \Gamma \to \Gamma$ such that
    \begin{equation*}
      v(f(x)) = g(v(x)).
    \end{equation*}
  \item If $k^\times$ has no exotic automorphisms or $\Gamma$
    has no exotic automorphisms, then $g(x) = q \cdot x$ for some
    rational number $q$.
  \item \label{chk3} If $g(x) = q \cdot x$ for some rational number $q$, then $q =
    n/p^m$ where $p$ is the characteristic exponent of $\Mm$ and $n,m$
    are integers.  Moreover, $f(x) = x^{n/p^m}$.
  \end{enumerate}
\end{theorem}
\begin{proof}
  For (1) and (2), Lemma~\ref{asymptotics} gives an endomorphism $g$ and a
  constant $c$ such that
  \begin{equation*}
    v(f(x)) = g(v(x)) + c \tag{$\ast$}
  \end{equation*}
  for all sufficiently small $x$.  Moreover, $g$ has the form $q \cdot
  x$ unless both $\Gamma$ and $k^\times$ have exotic automorphisms.

  Fix some $a \in \Mm^\times$.  Take $x$ so small that ($\ast$) holds
  for both $x$ and $ax$.  Then
  \begin{align*}
    v(f(ax)) &= g(v(ax)) + c \\
    v(f(x)) &= g(v(x)) + c \\
    \text{ and so } v(f(a)) &= g(v(a))
  \end{align*}
  by subtracting the second line from the first.  This proves (1) and (2).

  For part (3), first suppose that $g(x) = q \cdot x$ for some $q$ of
  the form $n/p^m$.  Then $v(f(x)) = q \cdot v(x) = v(x^q)$.  We must
  show that $f(x) = x^q$.  Replacing $f(x)$ with $f(x)/x^q$, we may
  assume $q = 0$, so that $v(f(x)) = 0$ for all $x$.  We must show
  that $f(x) = 1$.  Note that $f$ is a definable homomorphism
  $\Mm^\times \to \Oo^\times$.  By Theorem~\ref{limits-exist},
  \begin{equation*}
    \lim_{x \to 0} f(x) = c
  \end{equation*}
  for some $c \in \Mm \cup \{\infty\}$.  But $f(x)$ lies in the clopen
  set $\Oo^\times$, so in fact $c \in \Oo^\times$.  For any $a \in
  \Mm^\times$, we have
  \begin{equation*}
    c = \lim_{x \to 0} f(ax) = \lim_{x \to 0} f(a)f(x) = f(a) \cdot c.
  \end{equation*}
  This is impossible unless $f(a) = 1$.

  Next suppose that $g(x) = q \cdot x$ where $q$ does \emph{not} have
  the desired form.  We must obtain a contradiction.  Write $q$ in
  lowest terms as $n/(p^m \ell)$, where $\ell > 1$ is prime to $p$.
  Note that
  \begin{equation*}
    v(f(x)^\ell) = \ell \cdot q \cdot v(x) = (n/p^m) \cdot v(x).
  \end{equation*}
  By the case described above, we must have $f(x)^\ell = x^{n/p^m}$.
  Since $f$ is a multiplicative endomorphism,
  \begin{equation*}
    f(x^\ell) = x^{n/p^m} \text{ for any } x \in \Mm^\times.
  \end{equation*}
  Taking $x$ to be a primitive $\ell$th root of unity, we get a
  contradiction (the left side is $1$ and the right side is not).
\end{proof}
As a corollary, a definable endomorphism of $\Mm^\times$ is
determined by the induced map on $\Gamma$:
\begin{corollary} \label{napeequa}
  Let $h, h'$ be two definable endomorphisms of $\Mm^\times$.  Let $g$
  and $g'$ be the corresponding definable endomorphisms of
  $(\Gamma,+)$, so that
  \begin{gather*}
    v(h(x)) = g(v(x)) \\
    v(h'(x)) = g'(v(x))
  \end{gather*}
  If $g = g'$, then $h = h'$.
\end{corollary}
\begin{proof}
  Let $h''(x) = h(x)/h'(x)$.  Then $h''$ is a definable endomorphism
  of $\Mm^\times$, and the corresponding definable endomorphism of
  $\Gamma$ is $g''(x) = g(x)-g'(x) = 0$.  By
  Theorem~\ref{chikamin}(\ref{chk3}), with $q=0$, we see that $h''(x)
  = x^0 = 1$, so that $h(x) = h'(x)$.
\end{proof}
We can also deduce asymptotic expansions of definable functions, under
some special assumptions:
\begin{theorem} \label{asymptotics-ckd}
  Suppose $\kx$ has no exotic automorphisms or $\Gamma$ has no
  exotic automorphisms, and suppose $\Mm$ has characteristic 0.  Let $X$
  be a punctured neighborhood of 0 and $f : X \to \Mm$ be a definable
  function.  Then there is an increasing sequence of integers $m_1 <
  m_2 < \cdots $ and elements $c_1, c_2, \ldots \in \Mm$ such that
  \begin{equation*}
    \lim_{x \to 0} \frac{f(x) - \sum_{n = 1}^\ell
      c_nx^{m_n}}{x^{m_\ell}} = 0 \text{ for each } \ell.
  \end{equation*}
  In other words,
  \begin{equation*}
    f(x) = \sum_{n = 1}^\ell c_nx^{m_n} + o(x^{m_\ell}) \text{ as } x \to 0,
  \end{equation*}
  so $f(x)$ has the asymptotic expansion $\sum_{n=1}^\infty
  c_nx^{m_n}$.
\end{theorem}
\begin{proof}
  By Theorem~\ref{twist} and Theorem~\ref{chikamin}(2-3), any non-zero
  definable function $f(x)$ on a punctured neighborhood of 0 has
  asymptotic behavior
  \begin{equation*}
    f(x) \sim c x^m \text{ as } x \to 0,
  \end{equation*}
  for some $c \in \Mm^\times$ and $m \in \Zz$.  The remainder $f(x) -
  c x^m$ is $o(x^m)$, so it has asymptotic expansion
  \begin{equation*}
    f(x) - c x^m \sim c' x^{m'}
  \end{equation*}
  for some greater integer $m' > m$.  Iterating this, we get the
  desired asymptotic expansion.
\end{proof}
Cubides Kovacsics and Delon proved Theorem~\ref{asymptotics-ckd} in
the special case where definable unary functions on $\Gamma$ are
eventually $\Qq$-linear, a stronger assumption than ``$\Gamma$ has no
exotic automorphisms'' \cite[Theorem~6.1]{CK-D}.

\section{Strengthening limits}
%
%
We need the following random technical fact.
\begin{theorem} \label{upgrade}
  Let $B \subseteq \Mm$ be a ball containing 0.  Let $U \subseteq \Mm$ be definable and open.  Let $f :
  U \times B \to \Mm$ be a definable function.  Suppose that
  \begin{equation*}
    \lim_{y \to 0} f(a,y) = f(a,0)
  \end{equation*}
  for every $a \in U$.  Then $f$ is continuous at $(a,0)$ for all
  but finitely many $a \in U$.
\end{theorem}
\begin{proof}
  Otherwise, $\{a \in U : f \text{ is discontinuous at } (a,0)\}$
  has non-empty interior by $C$-minimality.  Shrinking $U$, we may
  assume that $f$ is discontinuous at $(a,0)$ for every $a \in U$.
  By generic continuity, we may shrink $U$ and assume that the
  function $f(x,0)$ is continuous for $x \in U$.

  If $\epsilon \in \Gamma$ and $a \in U$, say that $a$ is
  \emph{$\epsilon$-good} if $v(f(x,y)-f(a,0)) > \epsilon$ for all
  $(x,y)$ in a neighborhood of $(a,0)$.  Otherwise, say that $a$ is
  \emph{$\epsilon$-bad}.  Every $a \in U$ is $\epsilon$-bad for
  sufficiently large $\epsilon$, or else $f$ would be continuous at
  $(a,0)$.

  Take distinct $a_1, a_2, \ldots \in U$.  Take $\epsilon_i$ such
  that $a_i$ is $\epsilon_i$-bad.  By saturation, there is $\epsilon
  \ge \epsilon_i$ for every $i$.  Then $a_i$ is $\epsilon$-bad for
  every $i$.  Therefore, the set of $\epsilon$-bad points is infinite,
  and contains a ball by $C$-minimality.  Shrinking $U$, we may assume
  that every $a \in U$ is $\epsilon$-bad (for some fixed
  $\epsilon$).

  Fix this $\epsilon$.  Say that $a \in U$ is \emph{compatible with}
  $\delta \in \Gamma$ if
  \begin{equation*}
    v(y) > \delta \implies v(f(a,y)-f(a,0)) > \epsilon.
  \end{equation*}
  The fact that $\lim_{y \to 0} f(a,y) = f(a,0)$ implies that every $a
  \in U$ is compatible with all sufficiently large $\delta$.  Take
  distinct $a_1, a_2, \ldots$ in $U$.  Take $\delta_i$ compatible
  with $a_i$.  Take $\delta$ greater than every $\delta_i$.  Then
  every $a_i$ is compatible with $\delta$.  The set of $a \in U$
  compatible with $\delta$ is definable and infinite, so it contains a
  ball.  Shrinking $U$, we may assume that every $a \in U$ is
  compatible with $\delta$.

  Fix some $a_0 \in U$.  By continuity of $f(x,0)$, there is a
  ball $a_0 \in B' \subseteq U$ such that
  \begin{equation*}
    a \in B' \implies v(f(a,0)-f(a_0,0)) > \epsilon.
  \end{equation*}
  If $a \in B'$ and $v(y) > \delta$, then
  \begin{gather*}
    v(f(a,0)-f(a_0,0)) > \epsilon \\
    v(f(a,y)-f(a,0)) > \epsilon \\
    v(f(a,y)-f(a_0,0)) > \epsilon.
  \end{gather*}
  The second line holds because $a$ is compatible with $\delta$.  The
  third line holds by the ultrametric inequality.  If $B(\gamma)$
  denotes the open ball around 0 of radius $\gamma$, then we have just
  shown that
  \begin{equation*}
    (a,y) \in B' \times B(\delta) \implies v(f(a,y)-f(a_0,0)) > \epsilon.
  \end{equation*}
  Therefore $a_0$ is $\epsilon$-good, contradicting the fact that
  every point in $U$ is $\epsilon$-bad.
\end{proof}
As an example, if $f$ is differentiable, then $f$ is strictly
differentiable at almost all points:
\begin{theorem} \label{strong-diff}
  Let $B$ be a ball and $f : B \to \Mm$ be a differentiable definable
  function.  Then
  \begin{equation*}
    \lim_{(x,y) \to (a,a)} \frac{f(x)-f(y)}{x-y} = f'(a)
  \end{equation*}
  for all but finitely many $a \in B$.
\end{theorem}
\begin{proof}
  Translating, we may assume that $0 \in B$.  Then $B$ is a definable
  subgroup of $(\Mm,+)$, which simplifies some of the notation.  Consider the definable function
  \begin{gather*}
    g : B \times B \to \Mm \\
    g(a,y) = 
    \begin{cases}
      \frac{f(a+y)-f(a)}{y} & \text{ if } y \ne 0 \\
      f'(a) & \text{ if } y = 0.
    \end{cases}
  \end{gather*}
  Then $\lim_{y \to 0} g(a,y) = f'(a) = g(a,0)$.
  Theorem~\ref{upgrade} shows that for almost all $a \in B$,
  \begin{equation*}
    \lim_{(x,y) \to (a,0)} g(x,y) = g(a,0) = f'(a).
  \end{equation*}
  Restricting to $y \ne 0$, this implies
  \begin{equation*}
    \lim_{(x,y) \to (a,0)} \frac{f(x+y)-f(y)}{y} = f'(a).
  \end{equation*}
  Modulo a change of variables, this is what we want.
\end{proof}

\section{Generic differentiability}

\begin{theorem}\label{diff-thm}
  Suppose $\characteristic(\Mm) = 0$ and $\Mm$ is definably complete.
  If $B \subseteq \Mm$ is a ball and $f : B \to \Mm$ is definable,
  then there is a smaller ball $B' \subseteq B$ on which $f$ is
  differentiable.

  Consequently, $f$ is differentiable at all but finitely many points
  of $B$.
\end{theorem}
\begin{proof}
  By the proof of Lemma~\ref{shuksan}, we may shrink $B$ and assume that $f$ is continuous and that
  \begin{equation*}
    v(f(x)-f(y)) = g(v(x-y)) + c \qquad\qquad\text{ for } x,y \in B
  \end{equation*}
  for some increasing endomorphism $g : \Gamma \to \Gamma$ and constant
  $c$.
  For any fixed $a \in B$, Theorem~\ref{twist} gives a definable
  endomorphism $h_a : \Mm^\times \to \Mm^\times$ and constant $c_a \in
  \Mm^\times$ such that
  \begin{equation*}
    \lim_{x \to a} \frac{f(x)-f(a)}{h_a(x-a)} = c_a.
  \end{equation*}
  Theorem~\ref{chikamin} gives a definable endomorphism $g_a : \Gamma
  \to \Gamma$ such that $v(h_a(x)) = g_a(v(x))$ for sufficiently small
  $x$.  Then for $x$ sufficiently close to $a$, we have
  \begin{equation*}
    v(f(x)-f(a)) = v(c_a) + v(h_a(x-a)) = v(c_a) + g_a(v(x-a))
  \end{equation*}
  but also
  \begin{equation*}
    v(f(x)-f(a)) = c + g(v(x-a)).
  \end{equation*}
  Therefore $g = g_a$ and $c = v(c_a)$.  In particular, $g_a$ doesn't
  depend on $a$.  By Corollary~\ref{napeequa}, $h_a$ doesn't depend on
  $a$.  Therefore, there is a fixed definable endomorphism $h :
  \Mm^\times \to \Mm^\times$ such that $h_a = h$ and
  \begin{equation*}
    \lim_{x \to a} \frac{f(x)-f(a)}{h(x-a)} = c_a \in \Mm^\times \tag{$\ast$}
  \end{equation*}
  for all $a \in B$.  By the proof of Theorem~\ref{strong-diff}, we
  even get
  \begin{equation*}
    \lim_{(x,y) \to (a,a)} \frac{f(x)-f(y)}{h(x-y)} = c_a \in
    \Mm^\times \tag{$\dag$}
  \end{equation*}
  for almost all $a \in B$.  Fix such an $a$.
  \begin{claim}
    If we extend $h$ to a function $\Mm \to \Mm$ by setting $h(0)=0$,
    then $h(u+v) = h(u)+h(v)$ for $u,v \in \Mm$.
  \end{claim}
  \begin{proof}
    If $u,v \in \Mm$ are distinct, then
    \begin{equation*}
    \lim_{y \to 0} \frac{f(a+uy)-f(a+vy)}{h(y)} = h(u-v) \lim_{y \to 0} \frac{f(a+uy)-f(a+vy)}{h((a+uy)-(a+vy))} = h(u-v)c_a \tag{\ddag}
  \end{equation*}
  by ($\dag$), because $h((a+uy)-(a+vy)) = h((u-v)y) = h(u-v)h(y)$.
  Equation ($\ddag$) also holds when $u=v$, since then both sides are
  zero.

  Now letting $u$ and $v$ be arbitrary, we have
  \begin{align*}
    h(u)c_a + h(v)c_a &= \lim_{y \to 0} \frac{f(a+(u+v)y)-f(a+vy)}{h(y)} + \lim_{y \to 0} \frac{f(a+vy)-f(a)}{h(y)} \\
    &= \lim_{y \to 0} \frac{f(a+(u+v)y)-f(a)}{h(y)} = h(u+v)c_a.
  \end{align*}
  Cancelling a factor of $c_a$, we see that $h(u+v) = h(u)+h(v)$.
  \end{proof}
  By the claim, it follows that $h$ is a field embedding $\Mm \to
  \Mm$.  Let $F = \{x \in \Mm : h(x) = x\}$.  Then $F$ is a definable
  subfield of $\Mm$.  Because $\characteristic(F) =
  \characteristic(\Mm) = 0$, $F$ must be infinite.  Then $\dpr(F) =
  1$.  So $[\Mm : F] = \dpr(\Mm)/\dpr(F) = 1/1 = 1$, and $F = \Mm$,
  which means that $h(x) = x$ for all $x \in \Mm$.  Then $h =
  \id_\Mm$, and finally ($\ast$) says that $f$ is differentiable.
\end{proof}
By Theorem~\ref{strong-diff}, we also get strict differentiability at
almost all points.

\subsection{Multi-variable generic differentiability and the inverse function theorem} \label{sec:for-assaf}
Using techniques from \cite[Sections~3 and 5]{wj-P-minimal}, we can
strengthen Theorem~\ref{diff-thm} to yield generic differentiability
of functions in several variables, as well as the inverse function
theorem.

First, we need an analogue of \cite[Lemma~3.1]{wj-P-minimal}:
\begin{fact} \label{baire}
  Let $\{D_a\}_{a \in X}$ be a definable chain of subsets of $\Mm^k$, meaning that $\{D_a\}_{a \in X}$ is a definable family, $D_a \subseteq \Mm^k$ for any $a \in X$, and $D_a \subseteq D_b$ or $D_b \subseteq D_a$ for any $a,b \in X$.  If $\dim(D_a) < k$ for all $a \in X$, then $\dim\left(\bigcup_{a \in X} D_a\right) < k$.
\end{fact}
The proof of \cite[Lemma~3.1]{wj-P-minimal} applies almost verbatim,
using Fact~3.2 about dp-minimal theories in \cite{wj-P-minimal}.
Alternatively, Fact~\ref{baire} is an instance of a general fact about
visceral theories \cite[Theorem~2.52(9)]{wj-visceral}: the dimension
of a definable filtered union $\bigcup_{a \in X} D_a$ equals the
maximum $\max_{a \in X} \dim(D_a)$.

Fact~\ref{baire} then allows one to upgrade Theorem~\ref{upgrade} to a
statement in several variables:
\begin{theorem} \label{upgrade2}
  Let $B \subseteq \Mm$ be a ball containing 0.  Let $U \subseteq
  \Mm^k$ be definable and open.  Let $f : U \times B \to \Mm$ be a
  definable function . Suppose that $\lim_{y \to 0} f(a,y) = f(a,0)$
  for every $a \in U$.  Then
  \begin{equation*}
    \dim \{a \in U : f \text{ is discontinuous at } (a,0)\} < k.
  \end{equation*}
\end{theorem}
(Compare with Theorem~\ref{upgrade} above, or
\cite[Lemma~3.4]{wj-P-minimal}.)
\begin{proof}
  We may assume $U$ is non-empty, so $\dim(U) = k$.  As in the proof
  of Theorem~\ref{upgrade}, we can assume that $f$ is discontinuous at
  $(a,0)$ for every $a \in U$, but $f(x,0)$ is continuous on $U$.  Say
  that $a \in U$ is \emph{$\epsilon$-good} if $v(f(x,y) - f(a,0)) >
  \epsilon$ for all $(x,y)$ in a neighborhood of $(a,0)$, and
  \emph{$\epsilon$-bad} otherwise.  As in the proof of Theorem~\ref{upgrade}, every $a \in U$ is $\epsilon$-bad
  for sufficiently large $\epsilon$.  Then
  \begin{equation*}
    U = \bigcup_{\epsilon} \{a \in U : a \text{ is $\epsilon$-bad}\},
  \end{equation*}
  so Fact~\ref{baire} gives an $\epsilon$ such that
  \begin{equation*}
    \dim \{a \in U : a \text{ is $\epsilon$-bad}\} = k.
  \end{equation*}
  Then the set $\{a \in U : a \text{ is $\epsilon$-bad}\}$ has
  non-empty interior.  Shrinking $U$, we may assume that every point
  in $U$ is $\epsilon$-bad.\footnote{This is exactly the argument used
  to prove Lemma~3.1(2) in \cite{wj-visceral}.  In fact, we could
  directly apply Lemma~3.1(2) here.}

  Say that $a \in U$ is \emph{compatible with} $\delta$ if
  \begin{equation*}
    v(y) > \delta \implies v(f(a,y) - f(a,0)) > \epsilon.
  \end{equation*}
  As in the proof of Theorem~\ref{upgrade}, every $a \in U$ is
  compatible with sufficiently large $\delta$.  By a similar argument
  to the above\footnote{Or another application of
  \cite[Lemma~3.1(2)]{wj-visceral}.}, we can shrink $U$ and arrange
  for every $a \in U$ to be compatible with some fixed $\delta \in
  \Gamma$.

  The remainder of the proof of Theorem~\ref{upgrade} now applies
  verbatim, to give a contradiction.
\end{proof}
\begin{theorem} \label{diff-thm2}
  Suppose $\characteristic(\Mm) = 0$ and $\Mm$ is definably complete.
  If $U \subseteq \Mm^n$ is definable and open and $f : U \to \Mm^m$
  is definable, then there is a definable open subset $U_0 \subseteq
  U$ such that $f$ is strictly differentiable on $U_0$, and $\dim(U
  \setminus U_0) < n$.
\end{theorem}
\begin{proof}
  The proof of \cite[Theorem~5.14]{wj-P-minimal} applies verbatim, with the following changes:
  \begin{itemize}
  \item Use Theorem~\ref{diff-thm} instead of
    \cite[Proposition~5.12]{wj-P-minimal}, for single variable generic
    differentiability.
  \item Use Theorem~\ref{upgrade2} rather than
    \cite[Lemma~3.4]{wj-P-minimal} for strengthening limits.  \qedhere
  \end{itemize}
\end{proof}
Next, we turn to the implicit function theorem.  We need a replacement
for the definable compactness used in \cite{wj-P-minimal}.  To state
it, we need some notation.  If $\bc = (c_1,\ldots,c_n) \in \Mm^n$ is a
tuple, let $v(\bc)$ denote $\min_{1 \le i \le n} v(c_i)$.  In $\Mm^n$,
a \emph{closed ball of radius $\gamma$} is a set of the form $\{x \in
\Mm^n : v(x-a) \ge \gamma\}$.  Equivalently, a closed ball of radius
$\gamma$ in $\Mm^n$ is a set of the form $\prod_{i=1}^n B_i$ with
$B_i$ a closed ball of radius $\gamma$ in $\Mm$.  If $\Mm$ is
definably complete, then so is $\Mm^n$, in the sense that any
definable chain of closed balls in $\Mm^n$ with radii tending to
$+\infty$ has non-empty intersection.
\begin{lemma} \label{silly-completeness}
  Suppose $\Mm$ is definably complete.  Let $\{D_\gamma\}_{\gamma \in
    \Gamma}$ be a definable family of non-empty closed subsets of
  $\Mm^n$, such that
  \begin{enumerate}
  \item $D_\gamma \supseteq D_{\gamma'}$ for $\gamma' >
    \gamma$.
  \item If $a, b \in D_\gamma$, then $v(a-b) \ge \gamma$.
  \end{enumerate}
  Then $\bigcap_{\gamma \in \Gamma} D_{\gamma} \ne
  \varnothing$.  
\end{lemma}
\begin{proof}
  Let $B_\gamma \subseteq \Mm^n$ be the unique closed ball of radius
  $\gamma$ containing $D_\gamma$.  By definable completeness, there is
  $a \in \bigcap_\gamma B_\gamma$.  If $a \in \bigcap_\gamma
  D_\gamma$, we are done.  Otherwise, $a \notin D_{\gamma_0}$ for some
  $\gamma_0$.  Since $D_{\gamma_0}$ is closed, there is $\gamma \ge
  \gamma_0$ such that
  \begin{equation*}
    b \in D_{\gamma_0} \implies v(a - b) < \gamma. \tag{$\dag$}
  \end{equation*}
  Take $b \in D_\gamma \subseteq B_\gamma \cap D_{\gamma_0}$.  Then
  $a,b \in B_\gamma$, so $v(a-b) \ge \gamma$, contradicting ($\dag$).
\end{proof}
Using this, we get an analogue of \cite[Lemma~5.15]{wj-P-minimal}.
\begin{lemma} \label{contract}
  Let $\gamma_0 > 0$ be a positive element of $\Gamma$.  Let $B
  \subseteq \Mm^n$ be a ball around 0.  Let $f : B \to B$ be a
  function such that
  \begin{equation*}
    v(f(x) - f(y)) \ge \gamma_0 + v(x-y) \text{ for } x,y \in B.
  \end{equation*}
  \begin{enumerate}
  \item There is a unique $x \in B$ such that $f(x) = x$.
  \item For any $c \in B$, there is a unique $x \in B$ such that $f(x)
    = x - c$.
  \end{enumerate}
\end{lemma}
The assumption here is slightly stronger than
\cite[Lemma~5.15]{wj-P-minimal}, which merely required $v(f(x) - f(y))
> v(x-y)$.
\begin{proof}
  As in the proof of \cite[Lemma~5.15]{wj-P-minimal}, we only need to
  prove the existence part of (1).  Let $g : B \to \Gamma$ be the
  function $g(x) = v(f(x) - x)$.  Like in the proof of
  \cite[Lemma~5.15]{wj-P-minimal}, we get
  \begin{equation*}
    g(f(x)) \ge \gamma_0 + g(x). \tag{$\ast$}
  \end{equation*}
  By o-minimality of $\Gamma$, the set $\{g(x) : x \in B\} \subseteq
  \Gamma$ has a supremum $\gamma \in \Gamma \cup \{+\infty\}$.  Then
  ($\ast$) implies $\gamma + \gamma_0 = \gamma$, and so $\gamma =
  +\infty$.  Therefore, for any $\epsilon \in \Gamma$, there is $x \in
  B$ with $g(x) = v(f(x) - x) \ge \epsilon$.

  Let $D_\epsilon = \{x \in B : v(f(x) - x) \ge \epsilon\}$.  Then
  $D_\epsilon$ is non-empty and definable.  The function $f$ is
  continuous because it is contracting, so $D_\epsilon$ is closed.
  Finally,
  \begin{equation*}
    x,y \in D_\epsilon \implies v(x-y) \ge \epsilon.
  \end{equation*}
  Otherwise, the three conditions
  \begin{gather*}
    v(f(x) - x) \ge \epsilon \\
    v(f(y) - y) \ge \epsilon \\
    v(x-y) < \epsilon
  \end{gather*}
  would imply $v(f(x)-f(y)) = v(x-y)$ by the ultrametric inequality,
  contradicting the fact $f$ is contracting.  Then
  Lemma~\ref{silly-completeness} gives $a \in \bigcap_{\epsilon \in
    \Gamma} D_\epsilon$.  So $a \in B$ and $v(f(a) - a) \ge \epsilon$
  for all $\epsilon$.  This implies $f(a) = a$.
\end{proof}
\begin{theorem}[Inverse function theorem] \label{inverse}
  Suppose $\Mm$ is definably complete.  Let $U \subseteq \Mm^n$ be a
  definable open set, and let $f : U \to \Mm^n$ be a strictly
  differentiable definable function.  Let $a \in U$ be a point such
  that the strict derivative $Df(a)$ (an $n \times n$ matrix) is
  invertible.  Then there are open neighborhoods $U_0 \ni a$ and $V_0
  \ni f(a)$ such that $f$ restricts to a homeomorphism $U_0 \to V_0$,
  and the inverse map $f^{-1} : V_0 \to U_0$ is strictly
  differentiable.
\end{theorem}
\begin{proof}
  Fix some $\gamma_0 > 0$.  Then the proofs of Lemma~5.16,
  Corollary~5.17, and Theorem~5.18 in \cite{wj-P-minimal} carry
  through verbatim with the following changes to the proof of
  Lemma~5.16:
  \begin{itemize}
  \item When applying \cite[Fact~5.3]{wj-P-minimal} in the proof of
    Lemma~5.16, take $\gamma = \gamma_0 > 0$ rather than $\gamma = 0$,
    to get the stronger contraction assumption needed in
    Lemma~\ref{contract}.
  \item Apply Lemma~\ref{contract}(2) in place of
    \cite[Lemma~5.15(2)]{wj-P-minimal}.  \qedhere
  \end{itemize}
\end{proof}
One can deduce a version of the implicit function theorem as a
corollary.  We leave the details as an exercise to the reader.

\section{Future directions}
In Section~\ref{sec-heart}, we proved \emph{some} facts about the
induced structure on $\RV$, but we did not fully classify definable
sets $D \subseteq \RV^2$, let alone definable sets $D \subseteq
\RV^n$.  It would be nice to see how far these arguments can be
pushed, hopefully making the ideas of Section~\ref{sec-vague} precise.

In light of the main results of this paper, an obvious question is
\begin{question}
  Is there a $C$-minimal field such that $\kx$ has exotic automorphisms?
\end{question}
The induced structure on $k$ can be an arbitrary strongly minimal
expansion of ACF, so the question is really the following:
\begin{question} \label{qw}
  Is there a strongly minimal expansion of ACF in which $k^\times$ has
  an exotic definable automorphism?  Equivalently, is there a strongly
  minimal structure $(k,+,\cdot,\sigma)$ where $\sigma$ is a
  non-algebraic multiplicative automorphism?
\end{question}
Question~\ref{qw} is a variant of a well-known open problem:
\begin{question} \label{qhard}
  Is there a strongly minimal expansion of ACF in which $k$ has an
  exotic definable field automorphism?  Equivalently, is there a
  strongly minimal structure $(k,+,\cdot,\sigma)$ where $\sigma$ is a
  non-algebraic field automorphism?
\end{question}
Question~\ref{qhard} is only interesting in positive characteristic;
it is easy to see that there are no exotic field automorphisms in
characteristic 0.  On the other hand, Question~\ref{qw} remains open
in characteristic 0.  I suspect that in characteristic 0, it may be
possible to construct exotic multiplicative automorphisms using
techniques similar to the ones used to produce bad fields in
\cite{bose} (but I could be wrong---I am far from an expert in this
subject).

Another natural question is whether we can generalize the proof of the
exchange property to other settings, such as weakly o-minimal fields,
or more generally, dp-minimal fields.  Some aspects of the proof
generalize naturally, such as the dp-rank calculations.  However, in
other parts of the proof we use o-minimality and strong minimality.
It's not clear what the analogue would be in other more general
settings.

Lastly, it would be worthwhile to compare the proof of generic
differentiability here to the proofs of generic differentiability in
V-minimal fields \cite[Corollary~5.17]{hk} and Hensel minimal fields
\cite[Lemma~5.3.5]{hens-min1}, and look for a common generalization.
Perhaps the assumptions of Hensel minimality can be weakened.

\begin{acknowledgment}
  The author was supported by the National Natural Science Foundation
  of China (Grant No.\@ 12101131) and the Ministry of Education of
  China (Grant No.\@ 22JJD110002).  The author would like to thank
  Fran\c{c}oise Delon, Frank Wagner, Tom Scanlon, and Silvain
  Rideau-Kikuchi for helpful information.  This paper grew out of a
  report given at the 17th Asian Logic Conference in Tianjin; the
  author would like to thank the organizers of the conference.
  Lastly, the author would like to thank Assaf Hasson for suggesting
  adding Subsections~\ref{other-assaf} and \ref{sec:for-assaf}.
\end{acknowledgment}

\bibliographystyle{plain} \bibliography{mybib}{}

\end{document}